\documentclass[11pt,reqno]{amsart}
\usepackage{amsmath, amssymb, amsthm}
\usepackage{url}
\usepackage{hyperref}
\usepackage{tikz}
\usepackage{ytableau}
\usepackage[utf8]{inputenc}
\usepackage{xcolor}
\usepackage{tikz-cd}
\usepackage{hyperref}
\usepackage{nccmath}

\usepackage{youngtab}

\usetikzlibrary{decorations.pathreplacing}

\usetikzlibrary{decorations.pathreplacing}

\renewcommand{\(}{\left\(}
\renewcommand{\)}{\right\)}
\renewcommand{\[}{\left\[}
\renewcommand{\]}{\right\]}

\numberwithin{equation}{section}
\theoremstyle{plain}
\newtheorem{theorem}{Theorem}[section]
\newtheorem{lemma}[theorem]{Lemma}

\newtheorem{remark}[theorem]{Remark}

\newtheorem{definition}[theorem]{Definition}

\newtheorem{problem}[theorem]{Problem}

\newtheorem{corollary}[theorem]{Corollary}
\newtheorem{proposition}[theorem]{Proposition}

\makeatletter
\def\proof{\@ifnextchar[{\@oproof}{\@nproof}}
\def\@oproof[#1][#2]{\trivlist\item[\hskip\labelsep\textit{#2 Proof of\
		#1.}~]\ignorespaces}
\def\@nproof{\trivlist\item[\hskip\labelsep\textit{Proof.}~]\ignorespaces}

\makeatother

\begin{document}
	
	\title[GARGI MUKHERJEE]{Asymptotic Growth of $(-1)^{r} {\Delta}^r \log \sqrt[n]{\overline{p}(n)/n^{\alpha}}$ and the Reverse Higher Order Tur\'an Inequalities for $\sqrt[n]{\overline{p}(n)/n^{\alpha}}$} 
	
	\author{Gargi Mukherjee}
	\address{School of Mathematical Sciences,
		National Institute of Science Education and Research, Bhubaneswar, An OCC of Homi Bhabha National Institute,  Via- Jatni, Khurda, Odisha- 752050, India}
	\email{gargi@niser.ac.in}
	\maketitle
	
	\begin{abstract}
		Let $\overline{p}(n)$ denote the overpartition function. In this paper, we study the asymptotic growth of finite difference of logarithm of $\sqrt[n]{\overline{p}(n)/n^{\alpha}}$ for $\alpha$ being a non-negative real number, namely $(-1)^{r}\Delta^r \log \sqrt[n]{\overline{p}(n)/n^{\alpha}}$ by presenting an inequality of it with a symmetric upper and lower bound. Consequently, we arrive at log-convexity of $\sqrt[n]{\overline{p}(n)}$ and $\sqrt[n]{\overline{p}(n)/n}$, previously studied by the author. The another main objective of this paper is to introduce the notion of the reverse higher order Tur\'{a}n inequalities and we prove this for $\sqrt[n]{\overline{p}(n)/n^{\alpha}}$, which not only generalize the study of Sun, Chen, and Zheng but also depicts the non real-rootedness of the Jensen polynomial associated with the sequence mentioned before.

	\end{abstract}
	
	\hspace{0.65 cm} \textbf{2020 Mathematics Subject Classifications.}  Primary 05A20, 11N37, 11B68.\\
	\vspace{0.3 cm}
	\hspace{0.8 cm} \textbf{Keywords.} Overpartition, Finite difference, Reverse Higher
	order Tur\'{a}n inequalities.
	
	\section{Introduction}\label{intro}
		A positive sequence $\{a_n\}_{n \geq 0}$ is said to be $\log$-concave (resp. $\log$-convex) if for all $n \geq 1$, $a^2_n\geq a_{n-1}a_{n+1}$ (resp. $a^2_n\leq a_{n-1}a_{n+1}$), and it is said to be strictly $\log$-concave (resp. strictly $\log$-convex) if the inequality is strict. The binomial coefficients, the Stirling numbers, etc. (resp. the Catalan numbers, the Motzkin numbers etc.) (see \cite{Stanley, Brenti,Liu}) are some well known examples of $\log$-concave (resp. $\log$-convex) sequences.
		
	 Study on $\log$-concavity property of the partition function was initiated by Chen \cite{Chentalk-2010} and Nicolas \cite{Nicolas} independently. A partition of a positive integer $n$ is a non-increasing sequence of positive integers whose sum is $n$ and $p(n)$ denotes the number of partitions of $n$. DeSalvo and Pak \cite{DeSalvo-Pak-2015} settled Chen's conjecture by proving $\log$-concavity of $\{p(n)\}_{n\geq26}$ using the Hardy-Ramanujan-Rademacher formula \cite{RamanujanHardy, Rademacher} and the error estimation due to Lehmer \cite{Lehmer}.
		
		Similar phenomena can also be found for the overpartition function. An overpartition of $n$, a generalization of partitions introduced by Corteel and Lovejoy \cite{Corteel-Lovejoy-2004}, is a nonincreasing sequence of natural numbers whose sum is $n$ in which the first occurrence of a number
		may be overlined and $\overline{p}(n)$ denotes the number of overpartitions of $n$. For convenience, define $\overline{p}(0)=1$. For example, there are $8$ overpartitions of $3$ enumerated by $3, \overline{3}, 2+1, \overline{2}+1, 2+\overline{1}, \overline{2}+\overline{1},1+1+1, \overline{1}+1+1$. Analogous to Hardy-Ramanujan-Rademacher series for partition function, due to Zuckerman \cite{Zuckerman}, we have 
		\begin{equation}\label{Zuckerman}
			\overline{p}(n)=\frac{1}{2\pi}\underset{2 \nmid k}{\sum_{k=1}^{\infty}}\sqrt{k}\underset{(h,k)=1}{\sum_{h=0}^{k-1}}\dfrac{\omega(h,k)^2}{\omega(2h,k)}e^{-\frac{2\pi i n h}{k}}\dfrac{d}{dn}\biggl(\dfrac{\sinh \frac{\pi \sqrt{n}}{k}}{\sqrt{n}}\biggr),
		\end{equation}
		where
		\begin{equation*}
			\omega(h,k)=\text{exp}\Biggl(\pi i \sum_{r=1}^{k-1}\dfrac{r}{k}\biggl(\dfrac{hr}{k}-\biggl\lfloor\dfrac{hr}{k}\biggr\rfloor-\dfrac{1}{2}\biggr)\Biggr)
		\end{equation*}
		for $(h,k)\in \mathbb{Z}_{\geq 0}\times\mathbb{Z}_{\geq 1}$. Engel \cite{Engel-2017} estimated the error term, denoted by $R_2(n,N)$ of \eqref{Zuckerman} given by
		\begin{equation}\label{Engel1}
			\overline{p}(n)=\frac{1}{2\pi}\underset{2 \nmid k}{\sum_{k=1}^{N}}\sqrt{k}\underset{(h,k)=1}{\sum_{h=0}^{k-1}}\dfrac{\omega(h,k)^2}{\omega(2h,k)}e^{-\frac{2\pi i n h}{k}}\dfrac{d}{dn}\biggl(\dfrac{\sinh \frac{\pi \sqrt{n}}{k}}{\sqrt{n}}\biggr)+R_{2}(n,N)
		\end{equation}
		with
		\begin{equation}\label{Engel2}
			\bigl|R_{2}(n,N)\bigr|< \dfrac{N^{5/2}}{\pi n^{3/2}} \sinh \biggl(\dfrac{\pi \sqrt{n}}{N}\biggr).
		\end{equation}
	and proved that $\{\overline{p}(n)\}_{n \geq 2}$ is $\log$-concave by \eqref{Engel1} and \eqref{Engel2} with $N=3$.
		
Log-convexity for a certain class of combinatorial sequences has been recorded in \cite{Liu}. The $\log$-convexity of $p(n)$ was conjectured by Sun \cite{Sun}, which was settled later by Chen and Zheng \cite{Chen and Zheng}. Moreover, Chen and Zheng showed that $\{\sqrt[n]{p(n)/n^{\alpha}}\}_{n \geq n(\alpha)}$ is $\log$-convex for any real $\alpha$. Following a similar line of work with a more general setting, the $\log$-convexity of $\{\sqrt[n]{\overline{p}(n)/n^{\alpha}}\}_{n \geq n_1(\alpha)}$ and the asymptotic growth of ${\Delta}^2 \log \sqrt[n]{\overline{p}(n)/n^{\alpha}}$ were achieved through finding a symmetric upper and lower bound of ${\Delta}^2 \log \sqrt[n]{\overline{p}(n)/n^{\alpha}}$ in \cite{M1}. This study was ended up by raising the following problem; \\
		Define $r_{\alpha}(n):=\sqrt[n]{\overline{p}(n)/n^{\alpha}}$.
		\begin{problem}\cite[Problem 3.1]{M1}\label{M1}
			Let $\alpha$ be a non-negative real number. Then for each $r\geq 1$, does there exists a positive integer $N(r,\alpha)$ so that for all $n \geq N(r,\alpha)$, one can obtain both upper bound and lower bound of $(-1)^r \Delta^r \log r_{\alpha}(n)$ that finally shows the asymptotic growth of $(-1)^r \Delta^r \log r_{\alpha}(n)$ as $n$ tends to infinity?
		\end{problem}
A primary motivation to consider the Problem \ref{M1} is to pursue an extensive study the higher order differences of the logarithm of $\sqrt[n]{\overline{p}(n)/n^{\alpha}}$ from an inequality perspective so as to establish its asymptotic growth. Analogous to the work of Chen, Wang, and Xie \cite{Chen2} in this direction for $p(n)$, Wang, Xie, and Zhang \cite{WXZ} showed a similar behavior for $\overline{p}(n)$. But here we will observe that the similar phenomena also occurs in the case of a $\log$-convex sequence $\sqrt[n]{\overline{p}(n)/n^{\alpha}}$.
\newpage
\begin{definition}\label{newdef}
		\begin{eqnarray}\label{def1}
		N_0(m) &:=&
		\begin{cases}
		1, &\quad \text{if}\ m=1,\\
		2m \log m-m \log \log m, & \quad \text{if}\ m \geq 2.
		\end{cases}
		\end{eqnarray}
		\end{definition}
	\begin{definition}\label{newdef1}
	Following \cite[A000254]{Sloane}, we define the sequence $\{S_n\}_{n \geq 0}$\footnote[1]{Enumerates the number of permutations of $n+1$ elements with exactly two cycles.} as
	$$S_0=0\ \ \text{and}\ \ S_n=nS_{n-1}+(n-1)!.$$
	\end{definition}	
		\begin{definition}\label{Def} For $r \in \mathbb{Z}_{\geq 2}$ and $\alpha \in \mathbb{R}_{\geq 0}$, we define
			\begingroup
			\allowdisplaybreaks
			\begin{eqnarray}\label{def2}
				N_1(r)&:=& \max\Biggl\{85, \Biggl\lceil \dfrac{4}{\pi^2}N^2_0(2r+2)\Biggr\rceil \Biggr\},\\\label{def3}
				C(r)&:=&\pi \Bigl(\dfrac{1}{2}\Bigr)_{r},\\\label{def4}
				C(r,\alpha)&:=&(\alpha+1)r!+(r-1)!\log 8+1,\\\label{def5}
				C_2(r)&:=& \sum_{k=0}^{2r-2}\dfrac{1}{(k+1)^2\pi^{k+1}}\Bigl(\dfrac{k+1}{2}\Bigr)_r \dfrac{1}{r^k}+\dfrac{2r}{10^r},\\\label{def6}
				N_2(r)&:=& \max \Biggl\{\Biggl \lceil 2 \Bigl(C(r)\Bigr)^2 \Biggl \rceil, \Biggl \lceil \Biggl(\dfrac{2^{r+1}}{(r-1)! \log 8}\Biggr)^2 \Biggr \rceil \Biggr\},\\\label{def7}
				N_3(r,\alpha)&:=&\max \Biggl\{5505,4r^2, \Biggl \lceil \Biggl(4r^2C(r)+C_2(r)\Biggr)^\frac{4}{3}\Biggr \rceil, \Biggl\lceil \Biggl(\dfrac{C(r,\alpha)}{C(r)}\Biggr)^4 \Biggr\rceil \Biggr\},\\\label{def8}
				N(r,\alpha)&:=&\max\Biggl\{N_1(r),\Bigl\lceil e^{S_r/r!}\Bigr\rceil, N_2(r),N_3(r,\alpha)\Biggr\},\\\label{defnew9}
				\widetilde{N_1}(\alpha)&:=&\max\Bigl\{\lceil \sqrt{2\alpha}\rceil, \Bigl\lceil\sqrt[3]{3\alpha}\Bigr\rceil,\Bigl\lceil\sqrt[4]{11\alpha/3}\Bigr\rceil,\Bigl\lceil\sqrt[5]{10\alpha}\Bigr\rceil \Bigr\},\\\label{defnew10}
				\widetilde{N_2}(\alpha)&:=& \max\Bigl\{2, \lceil \sqrt{2\alpha}\rceil, \Bigl\lceil\sqrt[3]{3\alpha}\Bigr\rceil,\Bigl\lceil\sqrt[4]{11\alpha/3}\Bigr\rceil,\Bigl\lceil\sqrt[5]{12\alpha}\Bigr\rceil \Bigr\},\\\label{defnew11}
				\widetilde{N_3}(\alpha)&:=& \max\Bigl\{5505, \Bigl\lceil \Bigl(\frac{4 \alpha}{\alpha^2 +1}\Bigr)^{4/3} \Bigr\rceil, \Bigl\lceil (4 \alpha)^{4/11} \Bigr\rceil \Bigr\},\\\label{defnew12}
				\widetilde{N}(\alpha)&:=& \underset{1\leq i \leq 3}{\max}\Bigl\{\widetilde{N_i}(\alpha)\Bigr\}.
			\end{eqnarray}
			\endgroup
		\end{definition}
		\begin{theorem}\label{mainresult2}
		Let $N(r,\alpha)$ be as in \eqref{def8}. For $r \in \mathbb{Z}_{\geq 2}$ and $n \geq N(r, \alpha)$,
			\begin{equation}\label{mainresult2eqn}
				0<\log \Biggl(1+\dfrac{C(r)}{n^{r+1/2}}-\dfrac{C(r,\alpha)}{n^{r+3/4}}\Biggr)<(-1)^{r}\Delta^r \log r_{\alpha}(n) <\log \Biggl(1+\dfrac{C(r)}{n^{r+1/2}}\Biggr),
			\end{equation}
			where $C(r)$ and $C(r,\alpha)$ are defined in \eqref{def3} and \eqref{def4} respectively.
		\end{theorem}
			\begin{corollary}\label{Cor1}
			For $r \in \mathbb{Z}_{\geq 2}$,
			\begin{equation}\label{coreqn1}
				\lim_{n\rightarrow \infty} n^{r+1/2}(-1)^{r}\Delta^r \log \sqrt[n]{\overline{p}(n)/n^{\alpha}}= \pi \Bigl(\dfrac{1}{2}\Bigr)_{r}.
			\end{equation}
		\end{corollary}
		\begin{proof}
			Multiplying both sides of  \eqref{mainresult2eqn} by $n^{r+1/2}$ and taking limits as $n$ tends to infinity, we obtain \eqref{coreqn1}.
		\end{proof}
\begin{remark}\label{newremark}
	Theorem \ref{mainresult2} and Corollary \ref{Cor1} provide an explicit answer to Problem \ref{M1}. Also note that for $r=2$, \eqref{coreqn1} is documented in \cite[Corollary 1.5]{M1}.
\end{remark}

\begin{corollary}\cite[Corollary 1.3]{M1}\label{Cor3}
For $n \geq 4$, $r_0(n)^2< r_0(n-1)r_0(n+1)$ and for $n \geq 19$, $r_1(n)^2< r_1(n-1)r_1(n+1)$.	
\end{corollary}
	\begin{corollary}\label{cor2}
		The sequence $\bigl\{\sqrt[n]{\overline{p}(n)/n^{\alpha}}\bigr\}_{n \geq  N(2, \alpha)}$ is $\log$-convex.
	\end{corollary}
\begin{corollary}\label{Cor4} For $n \geq N(3,\alpha)$, $\Delta^3\log \sqrt[n]{\overline{p}(n)/(n)^{\alpha}}<0$. 
\end{corollary}
A sequence $\{a_n\}_{n \geq 0}$ is said to be ratio $\log$-concave (resp. ratio $\log$-convex) whenever the ratio sequence $\{{a_{n+1}}/{a_n}\}_{n\geq0}$ is $\log$-concave (resp. $\log$-convex). Thus  Corollary \ref{Cor4} is equivalent to state that the sequence $\{\sqrt[n]{\overline{p}(n)/n^{\alpha}}\}_{n\geq N(3,\alpha)}$ is ratio $\log$-concave. Following the notion of higher order $\log$-monotonic sequences \cite{Zhu}, Corollaries \ref{cor2} and \ref{Cor4} immediately lead to the order $2$ $\log$-monotonicity of $\{\sqrt[n]{\overline{p}(n)/n^{\alpha}}\}_{n \geq N(3,\alpha)}$.

A sequence $\{a_n\}_{n \geq 0}$ is said to satisfy the higher order Tur\'an inequalities if for $n \geq 1$,
\begin{equation*}
	4({a_n}^2 -a_{n-1}a_{n+1})({a_{n+1}}^2 -a_n a_{n+2})-(a_n a_{n+1}-a_{n-1} a_{n+2})^2 \geq 0.
\end{equation*}
 Chen, Jia, and Wang \cite{ChenJiaWang-2019} proved that $\{p(n)\}_{n \geq 95}$ satisfies the higher order Tur\'an inequalities. Besides $\log$-concavity of partitions, similar associated inequalities have been documented in \cite{Chentalk-2010,DeSalvo-Pak-2015,Chen2} and in many others. Liu and Zhang \cite{LiuZhang-2021} established a list of inequalities for overpartitions similar to those related with the partition function. Through the lens of study on roots of polynomials, $\log$-concavity of a sequence $\{a_n\}_{n\geq 0}$ is equivalent to the real-rootedness of the corresponding Jensen polynomial of degree 2 \begin{equation*}
 J^{2,n-1}_{a_n}(x) = \sum_{j=0}^{2}\binom{2}{j}a_{n-1+j}x^j.
 \end{equation*}
 Also note that the distinct real-rootedness of $J^{3,n-1}_{a}(x)=\sum_{j=0}^{3}\binom{3}{j}a_{n-1+j}x^j$ is equivalent to say that $\{a_n\}_{n \geq 0}$ satisfies the strict higher order Tur\'an inequalities. In general, $J^{d,n}_{a}(x):= \displaystyle\sum_{j=0}^{d}\binom{d}{j}a_{n+j}x^j$ is called the $n$-th Jensen polynomial of degree $d$ associated with $\{a_n\}_{n\geq 0}$. A detailed documentation on hyperbolicity of $J^{d,n}_{p}(x)$ can be found in the work of Griffin, Ono, Rollen, and Zagier \cite{GORZ}. Larson and Wagner \cite{LarsonWagner} computed an effective estimate of $N(d)$ such that for $n>N(d)$, $J^{d,n}_p(x)$ has all real roots. Following this theme, here we define the notion of reverse higher order Tur\'{a}n inequalities for a sequence.
 \begin{definition}\label{definition2}
	A sequence $\{a_n\}_{n \geq 0}$ is said to satisfy reverse higher order Tur\'an inequalities if for $n \geq 1$, 
	\begin{equation}\label{reverse}
		4({a_n}^2 -a_{n-1}a_{n+1})({a_{n+1}}^2 -a_n a_{n+2})-(a_n a_{n+1}-a_{n-1} a_{n+2})^2 \leq 0.
	\end{equation} 
\end{definition}
\begin{theorem}\label{mainresult3} 
For $n \geq N_T(\alpha)$\footnote{The cutoff $N_T(\alpha)$ given explicitly in \eqref{Gfinalcutoff}.},
	\begin{equation}\label{mainresult3eqn}
	\begin{split}
	4\Bigl({r_{\alpha}(n)}^2-r_{\alpha}(n-1)r_{\alpha}(n+1)\Bigr)\Bigl({r_{\alpha}(n+1)}^2
	&-r_{\alpha}(n)r_{\alpha}(n+2)\Bigr)\\
	&-\Bigl(r_{\alpha}(n)r_{\alpha}(n+1)-r_{\alpha}(n-1) r_{\alpha}(n+2)\Bigr)^2 < 0. 
	\end{split}
	\end{equation}
\end{theorem}
In view of the above discussion, Theorem \ref{mainresult3} can be rephrased as follows.
\begin{theorem}
For $n \geq N_T(\alpha)$, the Jensen polynomial $J^{3,n-1}_{r_{\alpha}}(x)$ has three distinct roots, among which one real and other two are complex conjugates.
\end{theorem}

We organize this paper as follows. Results associated with $(-1)^r \Delta^r \log r_{\alpha}(n)$ are given in Section \ref{sec1}. Before presenting the proof, after preparing a basic set up, preliminary lemmas (Lemmas \ref{lemfirst} and \ref{lem2}) in Subsection \ref{prelimlem}, and then by applying these two lemmas, we prove Theorem \ref{mainresult2} in Subsection \ref{prof1}. Study on the reverse higher Tur\'{a}n inequalities for $r_{\alpha}(n)$ is presented in Section \ref{sec2}. An upper bound estimation of $\dfrac{r_{\alpha}(n-1)r_{\alpha}(n+1)}{r_{\alpha}(n)^2}$ is carried out in Subsection \ref{subsecmain1}, and using that bound, we give a proof of Theorem \ref{mainresult3} in Subsection \ref{prof2}. Finally we conclude this paper by raising a question on infinite log-convexity of $r_{\alpha}(n)$ in Section \ref{conclusion}.
\section{Inequalities for $(-1)^r \Delta^r \log r_{\alpha}(n)$ }\label{sec1}
\subsection{Preliminary lemmas}\label{prelimlem}
Set $\widehat{\mu}(n)=\pi \sqrt n$. First we rewrite $\overline{p}(n),$ by setting $N=3$ in (\ref{Engel1}), as follows
\begin{equation}\label{eqn1}
	\overline{p}(n)=\widehat{T}(n) \Biggl(1+\dfrac{\widehat{R}(n)}{\widehat{T}(n)}\Biggr),
\end{equation}
where
\begin{eqnarray}\label{eqn2}
	\widehat{T}(n)&=& \dfrac{1}{8n}\Bigl(1-\dfrac{1}{\widehat{\mu}(n)}\Bigr)e^{\widehat{\mu}(n)},\\ \label{eqn3}
	\widehat{R}(n)&=& \dfrac{1}{8n}\Bigl(1+\dfrac{1}{\widehat{\mu}(n)}\Bigr)e^{-\widehat{\mu}(n)}+R_2(n,3).
\end{eqnarray}
Taking logarithm on both sides of (\ref{eqn1}), along with utilizing (\ref{eqn2}) and (\ref{eqn3}), we have
\begin{align*}
	\log r_{\alpha}(n) &= \dfrac{\log \overline{p}(n)}{n} - \dfrac{\alpha \log n}{n}\\
	&= \dfrac{\widehat{\mu}(n)}{n} -\dfrac{(\alpha +1) \log n}{n}+\dfrac{\log (\widehat{\mu}(n)-1)-\log \widehat{\mu}(n)}{n}-\dfrac{\log 8}{n} + \dfrac{1}{n}\log\  \Biggl(1+\dfrac{\widehat{R}(n)}{\widehat{T}(n)}\Biggr).
\end{align*}
Consequently, by $r$-fold application of $\Delta$ on $\log r_{\alpha}(n)$, it follows that
\begin{equation}\label{eqn4}
	(-1)^{r}\Delta^r \log r_{\alpha}(n)=H_{r,\alpha}(n)+G_r(n),
\end{equation}
where
\begin{eqnarray}\label{eqn5}
	H_{r,\alpha}(n) &=& (-1)^{r}\Delta^r \Biggl(\dfrac{\widehat{\mu}(n)}{n} -\dfrac{(\alpha +1) \log n}{n}+\dfrac{\log (\widehat{\mu}(n)-1)-\log \widehat{\mu}(n)}{n}-\dfrac{\log 8}{n}\Biggr),\\ \label{eqn6new}
	G_r(n) &=& (-1)^{r}\Delta^r \Biggl( \dfrac{1}{n} \log\  \Biggl(1+\dfrac{\widehat{R}(n)}{\widehat{T}(n)}\Biggr) \Biggr).
\end{eqnarray}
Then we have for $r \geq 1$,
\begin{equation}\label{eqn8new}
	H_{r,\alpha}(n)-|G_r(n)| \leq (-1)^{r}\Delta^r \log r_{\alpha}(n) \leq H_{r,\alpha}(n)+|G_r(n)|.
\end{equation}
Next, we give bounds for $H_{r,\alpha}(n)$ and $|G_r(n)|$ individually to get the estimates for $(-1)^{r}\Delta^r \log r_{\alpha}(n)$.
Applying \cite[Lemma 2.1]{GZZ} and following the sketch of proof of \cite[Lemma 2.3]{Mukherjee}, we obtain an upper bound for $|G_r(n)|$.
\begin{lemma}\label{lemfirst}
Let $N_1(r)$ be defined as in \eqref{def2}.	For all $n \geq N_1(r)$ and $r \geq 1$,
	\begin{equation}\label{lem1eqn1}
		|G_r(n)| < \dfrac{1}{n^{r+3/2}}.
	\end{equation}
\end{lemma}
Before we proceed for derivation of upper and lower bound of $H_{r,\alpha}(n)$, let us state the following proposition due to Odlyzko \cite{Odlyzko}.
\begin{proposition}\label{Odprop}
	Let $r$ be a positive integer. Suppose that $f(t)$ is a function with continuous derivatives for $t\geq 1$, and $(-1)^{k-1}f^{(k)}(t)>0$ for $k \geq 1$. Then for $r\geq 1$, $x\geq 1$,
	$$(-1)^{r-1}f^{(r)}(x+r)\leq (-1)^{r-1}\Delta^r f(x) \leq (-1)^{r-1}f^{(r)}(x).$$	
\end{proposition}
\begin{lemma}\label{lem2}
Let $S_r$ be as defined in Definition \ref{newdef1}. For $r \geq 2$ and $n \geq  \lceil e^{S_r/r!} \rceil$,
	\begin{equation}\label{lem2eqn1}
		L_{r,\alpha}(n)	< H_{r,\alpha}(n) <	U_{r,\alpha}(n) ,
	\end{equation}
	where
	\begin{equation}\label{neweqn1}
	U_{r,\alpha}(n):= \dfrac{C(r)}{n^{r+\frac{1}{2}}} -\sum_{k=1}^{\infty} \dfrac{2}{k^2 \pi^k} \Bigl(\dfrac{k}{2}\Bigr)_{r} \dfrac{1}{(n+r)^{\frac{k}{2}+r+1}}  - \dfrac{(\alpha+1) \bigl(r!\log(n+r)-S_r\bigr)}{(n+r)^{r+1}}-\dfrac{r!\log 8}{(n+r)^{r+1}},
	\end{equation}
	and
	\begin{equation}\label{neweqn2}
	L_{r,\alpha}(n):= \dfrac{C(r)}{(n+r)^{r+\frac{1}{2}}} -\sum_{k=1}^{\infty} \dfrac{2}{k^2 \pi^k} \Bigl(\dfrac{k}{2}\Bigr)_{r} \dfrac{1}{n^{\frac{k}{2}+r+1}}  - \dfrac{(\alpha+1) \bigl(r!\log n-S_r\bigr)}{n^{r+1}}-\dfrac{r!\log 8}{n^{r+1}}.
	\end{equation}
\end{lemma} 
\begin{proof}
	Rewrite \eqref{eqn5} as follows
	\begin{equation}\label{eqn6}
		H_{r,\alpha}(n)=(-1)^r \Delta^r \Biggl(\sum_{i=1}^{3} g_i(n)+g_{4,\alpha}(n)\Biggr),
	\end{equation} 
	where
$$g_1(n)=\dfrac{\widehat{\mu}(n)}{n},\ g_2(n)=\dfrac{\log (\widehat{\mu}(n)-1)-\log \widehat{\mu}(n)}{n},\ g_3(n)=-\dfrac{\log 8}{n},\ \text{and}\ \  g_{4,\alpha}(n)=-\dfrac{(\alpha +1) \log n}{n}.$$	
Taking $r$-th derivative for each of the above functions, we have 
\begin{equation*}
\begin{split}
(-1)^{r-1} g_1^{(r)}(n) &:= -\pi \Bigl(\dfrac{1}{2}\Bigr)_r \dfrac{1}{n^{r+\frac{1}{2}}} <0\ \ \text{for all}\ \  n \geq 1,\\
(-1)^{r-1} g_2^{(r)}(n)&:= \sum_{k=1}^{\infty} \dfrac{2}{k^2 \pi^k} \Bigl(\dfrac{k}{2}\Bigr)_{r} \dfrac{1}{n^{\frac{k}{2}+r+1}} >0\ \ \text{for all}\ \  n \geq 1,\\
(-1)^{r-1} g_3^{(r)}(n)&:=\dfrac{r!\log 8}{n^{r+1}} >0\ \ \text{for all}\ \ n \geq 1,\\
\text{and}\ \ 
(-1)^{r-1} g_{4,\alpha}^{(r)}(n)&:= \dfrac{(\alpha+1) (r!\log n-S_r)}{n^{r+1}} >0\ \ \text{for all}\ \ n \geq  \lceil e^{S_r/r!} \rceil.
\end{split}
\end{equation*}
Now applying Proposition \ref{Odprop} individually on each of the factors $\{g_i(n)\}_{1\leq i \leq 3}$ and $g_{4,\alpha}(n)$, we immediately arrive at \eqref{lem2eqn1}.
\end{proof}
\subsection{Proof of Theorem \ref{mainresult2}}\label{prof1}
From Lemmas \ref{lemfirst} and \ref{lem2}, for all $n \geq \max\Bigl\{N_1(r), \Bigl\lceil e^{S_r/r!} \Bigr\rceil \Bigr\}$, it follows that
\begin{equation*}
(-1)^r \Delta^{r} \log r_{\alpha}(n)<U_{r,\alpha}(n)+\dfrac{1}{n^{r+\frac{3}{2}}}.
\end{equation*}
Since for all $n \geq \max \Bigl\{r, \Bigl\lceil e^{ S_r/r!} \Bigr\rceil \Bigr\}$,
\begin{equation*}
\dfrac{(\alpha+1)(r!\log(n+r)-S_r)}{(n+r)^{r+1}} > 0\ \ \text{and}\ \ \sum_{k=1}^{\infty} \dfrac{2}{k^2 \pi^k} \Bigl(\dfrac{k}{2}\Bigr)_{r} \dfrac{1}{(n+r)^{\frac{k}{2}+r+1}} > 0,
\end{equation*}   
we have for all $n \geq \max\Bigl\{N_1(r), r, \Bigl\lceil e^{S_r/r!}\Bigr \rceil \Bigr\}$,
\begin{equation}\label{eqn8}
	(-1)^r \Delta^r \log r_{\alpha}(n)  <  \dfrac{C(r)}{n^{r+\frac{1}{2}}}  -\dfrac{r!\log 8}{(n+r)^{r+1}} + \dfrac{1}{n^{r+\frac{3}{2}}}. 
\end{equation}Recall that 
$N_2(r)=\max \Biggl\{\Biggl \lceil 2 \Bigl(C(r) \Bigr)^{\frac{2}{r}} \Biggr \rceil, \Biggl \lceil \Biggl( \dfrac{2^{r+1}}{r! \log 8} \Biggr)^2 \Biggr \rceil \Biggr\} \hspace*{0.2cm} \text{(cf. \eqref{def6})}.\\$ \\
We aim to show for all $n \geq 	N_2(r),$
\begin{equation}\label{eqn9}
	-\dfrac{r!\log 8}{(n+r)^{r+1}} + \dfrac{1}{n^{r+\frac{3}{2}}} <-\dfrac{1}{2} \Biggl(\dfrac{C(r)}{n^{r+\frac{1}{2}}}\Biggr)^2,
\end{equation}	
which is equivalent to prove 
\begin{equation*}
	\Biggl( \sqrt{n} r! \log 8 - 2^{r+1}\Biggr)n^r > 2^r C(r)^2.
\end{equation*}
Note that for all $n> \Bigl \lceil 2 \Bigl(C(r) \Bigr)^{\frac{2}{r}} \Bigr \rceil$,
\begin{equation}\label{eqn10}
	n^r > 2^r C(r)^2,
\end{equation}  
and 
\begin{equation}\label{eqn11}
	\Biggl( \sqrt{n} r! \log 8 - 2^{r+1}\Biggr)n^r>n^r \ \  \text{for all}\ \  n > \Biggl \lceil \Biggl( \dfrac{2^{r+1}}{r! \log 8} \Biggr)^2 \Biggr \rceil.
\end{equation}
Thus from \eqref{eqn10} and \eqref{eqn11}, we get \eqref{eqn9}.
Since for all positive integers $n$ and $r$, $C(r)>0$, summerizing \eqref{eqn8} and \eqref{eqn9}, we have for all $n \geq \max\Bigl\{N_1(r), \lceil e^{S_r/r!} \rceil, N_2(r)\Bigr\}$,
\begin{equation}\label{eqn12}
	(-1)^r \Delta^r \log r_{\alpha}(n) < \log \Biggl(1+ \dfrac{C(r)}{n^{r+\frac{1}{2}}}\Biggr).
\end{equation}
We further move on to show the lower bound given in Theorem \ref{mainresult2}.
From \eqref{lem1eqn1} and \eqref{lem2eqn1}, for all $n \geq \max\Bigl\{N_1(r), \Bigl\lceil e^{S_r/r!} \Bigr\rceil \Bigr\}$, we have
\begin{equation}\label{eqn13}
(-1)^r \Delta^{r} \log r_{\alpha}(n)>L_{r,\alpha}(n)-\dfrac{1}{n^{r+\frac{3}{2}}}.
\end{equation}
After rewriting the first factor in $L_{r,\alpha}(n)$ (cf. see \eqref{neweqn2}), we get
\begin{equation}\label{eqn14}
\dfrac{C(r)}{(n+r)^{r+\frac{1}{2}}}=\dfrac{C(r)}{n^{r+\frac{1}{2}}} \Biggl(1+\dfrac{r}{n}\Biggr)^{-(r+\frac{1}{2})}=\dfrac{C(r)}{n^{r+\frac{1}{2}}}+  \dfrac{C(r)}{n^{r+\frac{1}{2}}} \Biggl(\sum_{k=1}^{\infty} \binom{-\frac{2r+1}{2}}{k} \Bigl(\dfrac{r}{n}\Bigr)^k\Biggr).
\end{equation}
Clearly for all $(r,k)\in \mathbb{Z}_{\geq 1}\times \mathbb{Z}_{\geq 1}$, it follows that 
\begin{equation*}
\dfrac{4^r}{2 \sqrt{r}} \leq \binom{2r}{r} \leq \dfrac{4^r}{\sqrt{\pi r}},\ \binom{r+k+1}{r} \leq {(2r)}^{k+1},\ \text{and}\ \ \frac{2}{\sqrt{\pi}} \sqrt{\dfrac{r}{r+k}}<2.
\end{equation*}
Following a similar methodology devised in \cite[equation (2.21)]{Mukherjee}, we have for all $n \geq 4r^2$,
\begin{equation}\label{eqn15}
	\Biggl| \sum_{k=1}^{\infty} \binom{-\frac{2r+1}{2}}{k} \Bigl(\dfrac{r}{n}\Bigr)^k \Biggr| \leq \dfrac{4r^2}{n}.
\end{equation} 
Applying \eqref{eqn15} on \eqref{eqn14} we get for all $n \geq 4r^2$,
\begin{equation}\label{eqn16}
	\dfrac{C(r)}{(n+r)^{r+\frac{1}{2}}} > \dfrac{C(r)}{n^{r+\frac{1}{2}}}-\dfrac{4r^2C(r)}{n^{r+\frac{3}{2}}}. 
\end{equation}
To estimate the infinite series in $L_{r,\alpha}(n)$ (cf. \eqref{neweqn2}), we first split it as
\begin{eqnarray}
	\sum_{k=1}^{\infty} \dfrac{2}{k^2 \pi^k} \Bigl(\dfrac{k}{2}\Bigr)_{r} \dfrac{1}{n^{\frac{k}{2}+r+1}} &=& \dfrac{1}{n^{r+\frac{3}{2}}} \sum_{k=0}^{2r-2} \dfrac{2}{(k+1)^2 \pi^{k+1}} \Bigl(\dfrac{k+1}{2}\Bigr)_{r} \dfrac{1}{n^{\frac{k}{2}}} \nonumber +  \dfrac{1}{n^{r+\frac{3}{2}}} \sum_{k=2r}^{\infty} \dfrac{2}{k^2 \pi^{k}} \Bigl(\dfrac{k}{2}\Bigr)_{r} \dfrac{1}{n^{\frac{k-1}{2}}} \nonumber \\ & \underset{\forall n \geq r^2}{\leq} & \dfrac{1}{n^{r+\frac{3}{2}}}\underset{:=\widehat{C_2}(r)}{\underbrace{\sum_{k=0}^{2r-2} \dfrac{2}{(k+1)^2 \pi^{k+1}} \Bigl(\dfrac{k+1}{2}\Bigr)_{r} \dfrac{1}{r^k}}} +  \dfrac{r}{n^{r+\frac{3}{2}}} \underset{:=\widehat{S}(r)}{\underbrace{\sum_{k=2r}^{\infty} \dfrac{2}{k^2 \pi^{k}} \Bigl(\dfrac{k}{2}\Bigr)_{r} \dfrac{1}{r^{k}}}}.  \nonumber
\end{eqnarray}
In a similar fashion as carried out in \cite[equation (2.27)]{Mukherjee}, for all $n \geq r^2$, $\widehat{S}(r) < \dfrac{2}{10^r},$
and therefore, we have \begin{equation}\label{eqn17}
	\sum_{k=1}^{\infty} \dfrac{2}{k^2 \pi^k} \Bigl(\dfrac{k}{2}\Bigr)_{r} \dfrac{1}{n^{\frac{k}{2}+r+1}} < \dfrac{C_2(r)}{n^{r+\frac{3}{2}}}. 
\end{equation} 
Applying \eqref{eqn16} and \eqref{eqn17} to \eqref{eqn13}, we arrive at
\begingroup
\allowdisplaybreaks
\begin{eqnarray}\label{neweqn3}
	(-1)^r \Delta^{r} \log r_{\alpha}(n) &>&  \dfrac{C(r)}{n^{r+\frac{1}{2}}} -\dfrac{(\alpha+1) (r!\log n-S_r)}{n^{r+1}}  -\dfrac{r!\log 8}{n^{r+1}}   - \dfrac{4r^2 C(r)+C_2(r)+1}{n^{r+\frac{3}{2}}} \nonumber \\ & > &  \dfrac{C(r)}{n^{r+\frac{1}{2}}} -\dfrac{(\alpha+1) r!}{n^{r+\frac{3}{4}}}+\dfrac{S_r (\alpha +1)}{n^{r+1}}  -\dfrac{r!\log 8}{n^{r+1}}   - \dfrac{4r^2 C(r)+C_2(r)+1}{n^{r+\frac{3}{2}}}\nonumber \\ & & \hspace*{4cm}\Biggl(\text{since} \hspace*{0.2cm} \log n < n^{\frac{1}{4}} \hspace*{0.2cm} \text{for all} \hspace*{0.2cm} n \geq 5505\Biggr) \nonumber \\ & > & \dfrac{C(r)}{n^{r+\frac{1}{2}}} -\dfrac{C(r,\alpha)}{n^{r+\frac{3}{4}}} \nonumber \\ & & \Biggl(\text{Since} \hspace*{0.2cm} -\dfrac{4r^2 C(r)+C_2(r)}{n^{r+\frac{3}{2}}}> -\dfrac{1}{n^{r+\frac{3}{4}}} \hspace*{0.2cm} \text{for all} \hspace*{0.2cm} n \geq \Biggl \lceil (4r^2 C(r)+C_2(r))^{\frac{4}{3}} \Biggr \rceil \nonumber \\ & & \text{and} \hspace*{0.2cm} \dfrac{S_r(\alpha +1)}{n^{r+1}}-\dfrac{1}{n^{r+\frac{3}{2}}}>0 \hspace*{0.2cm} \text{for all} \hspace*{0.2cm} n \geq \Biggl \lceil \dfrac{1}{(S_r(\alpha+1))^2} \Biggr \rceil \Biggr) \nonumber \\ & > & \log \Biggl(1+\dfrac{C(r)}{n^{r+\frac{1}{2}}} -\dfrac{C(r,\alpha)}{n^{r+\frac{3}{4}}}\Biggr) \nonumber \\ & & \Biggl(\text{since} \hspace*{0.2cm} \dfrac{C(r)}{n^{r+\frac{1}{2}}} -\dfrac{C(r,\alpha)}{n^{r+\frac{3}{4}}}>0 \hspace*{0.2cm} \text{for all} \hspace*{0.2cm} n \geq \Biggl \lceil \Biggl(\dfrac{C(r,\alpha)}{C(r)}\Biggr)^4 \Biggr \rceil \Biggr). \\ \nonumber
\end{eqnarray}
\endgroup 
Recall that \begin{equation*}
	N_3(r, \alpha)= \max\Biggl\{5505, 4r^2, \Biggl \lceil (4r^2C(r)+C_2(r))^{\frac{4}{3}} \Biggr \rceil, \Biggl \lceil \Biggl(\dfrac{C(r,\alpha)}{C(r)}\Biggr)^4 \Biggr \rceil \Biggr\}  \text{(cf. \eqref{def7})}.
\end{equation*} 
From \eqref{neweqn3}, we have for all $n \geq \max\Bigl\{N_1(r), \Bigl\lceil e^{S_r/r!}\Bigr \rceil, N_3(r, \alpha)\Bigr\}$, 
\begin{equation}\label{eqn18}
	(-1)^r \Delta^{r} \log r_{\alpha}(n)>\log \Biggl(1+\dfrac{C(r)}{n^{r+\frac{1}{2}}} -\dfrac{C(r,\alpha)}{n^{r+\frac{3}{4}}}\Biggr).
\end{equation}
In view of \eqref{eqn12} and \eqref{eqn18}, for all $n \geq N(r, \alpha)$, the required estimates for $(-1)^r {\Delta}^r \log r_{\alpha}(n)$ as stated in \eqref{mainresult2eqn} has been proved.

\begin{remark}\label{newremark1}
Here we omit the proof of an inequality for $-\Delta \log r_{\alpha}(n)$ (the case $r=1$) due to its similarity with the proof of \cite[Theorem 1.6]{Mukherjee}. Applying equations \eqref{eqn5}-\eqref{eqn8new} with $r=1$ and Lemma \ref{lemfirst}, one can obtain an inequality similar to Theorem \ref{mainresult2}. 
\end{remark}
\section{Reverse higher order Tur\'an inequalities for $r_{\alpha}(n)$}\label{sec2}
Define $u_{\alpha}(n):=\dfrac{r_{\alpha}(n-1)r_{\alpha}(n+1)}{r_{\alpha}(n)^2}$.\subsection{Upper bound of $u_{\alpha}(n)$}\label{subsecmain1}
\begin{lemma}\label{lem3}
	For $n \geq 1$, 
	\begin{equation}\label{lem3eqn31}
		s_{+}(n)-\dfrac{120}{n^5}<\sqrt[n+1]{\dfrac{\overline{p}(n+1)}{\overline{p}(n)}}<s_{+}(n)+\dfrac{850}{n^5},
	\end{equation}
	where $s_+(n):=\displaystyle\sum_{i=0}^{9}\frac{a_i}{\sqrt{n}^i}$
	with
	\begin{equation*}
	\begin{split}
	&a_0=1, a_1=0, a_2=0, a_3=\dfrac{\pi}{2}, a_4=-1, a_5=\dfrac{4-5\pi^2}{8\pi}, a_6=\dfrac{4+12\pi^2+\pi^4}{8\pi^2}, a_7=\dfrac{8-14\pi^2+3\pi^4}{16\pi^3}\\
	&a_8=\dfrac{24-48\pi^2-52\pi^4-15\pi^6}{48\pi^4}, a_9=\dfrac{192-432\pi^2+360\pi^4+249\pi^6+8\pi^8}{384\pi^5}.
	\end{split}
	\end{equation*}
\end{lemma}
	\begin{proof}
		Applying \cite[eqn. (2.15)]{GZZ} with $m=10$, for all $n \geq 1297$, it follows that
		\begin{equation}\label{eqn20*}
			\Biggl(\frac{\widehat{T}(n+1)}{\widehat{T}(n)}\Biggr)^{\frac{1}{n+1}} \Biggl(1-4\frac{2^{10}}{{\widehat{\mu}(n)}^{10}}\Biggr)^{\frac{1}{n+1}}<\sqrt[n+1]{\dfrac{\overline{p}(n+1)}{\overline{p}(n)}}<\Biggl(\frac{\widehat{T}(n+1)}{\widehat{T}(n)}\Biggr)^{\frac{1}{n+1}} \Biggl(1+6\frac{2^{10}}{{\widehat{\mu}(n)}^{10}}\Biggr)^{\frac{1}{n+1}}.
		\end{equation}
		Using the decomposition for $\overline{p}(n)$ (cf. \eqref{eqn1}), we get
		\begin{equation*}
			\Biggl(\frac{\widehat{T}(n+1)}{\widehat{T}(n)}\Biggr)^{\frac{1}{n+1}}=\Biggl(\frac{\widehat{\mu}(n+1)-1}{\widehat{\mu}(n)-1}\Biggr)^{\frac{1}{n+1}} \cdot\Biggl(\frac{\widehat{\mu}(n)}{\widehat{\mu}(n+1)}\Biggr)^{\frac{3}{n+1}}\cdot e^{\frac{\widehat{\mu}(n+1)-\widehat{\mu}(n)}{n+1}}.
		\end{equation*}
	By Taylor expansion, we obtain for $n \geq 2$,
		\begin{equation}\label{eqn20}
			s^{(1)}_+(n)-\dfrac{20}{n^5}<\Biggl(\frac{\widehat{\mu}(n)}{\widehat{\mu}(n+1)}\Biggr)^{\frac{3}{n+1}}<s^{(1)}_+(n)+\dfrac{40}{n^5},
		\end{equation}
		where
		$$s^{(1)}_+(n):=1-\dfrac{3}{2n^2}+\dfrac{9}{4n^3}-\dfrac{13}{8n^4}.$$
	We decomposing the first factor of the ratio $\Bigl(\frac{\overline{T}(n+1)}{\overline{T}(n)}\Bigr)^{\frac{1}{n+1}}$ as
		\begin{equation*}
			\Biggl(\frac{\widehat{\mu}(n+1)-1}{\widehat{\mu}(n)-1}\Biggr)^{\frac{1}{n+1}}=\Biggl(1+\frac{1}{n}\Biggr)^{\frac{1}{2(n+1)}}\cdot\Biggl(1-\frac{1}{\widehat{\mu}(n+1)}\Biggr)^{\frac{1}{n+1}}\cdot\Biggl(1-\frac{1}{\widehat{\mu}(n)}\Biggr)^{-\frac{1}{n+1}},
		\end{equation*}
		and apply Taylor expansion on each of these three factors. We get for $n \geq 2$,
		\begin{equation}\label{eqn21}
		s^{(2,1)}_+(n)-\dfrac{7}{n^5}<	\Biggl(1+\frac{1}{n}\Biggr)^{\frac{1}{2(n+1)}}<s^{(2,1)}_+(n)+\dfrac{5}{n^5},
		\end{equation}
		where $$s^{(2,1)}_+(n):=1+\dfrac{1}{2n^2}-\dfrac{3}{4n^3}+\dfrac{25}{24n^4}.$$
	Note that for $n \geq 2$,
		\begin{equation}\label{eqn22}
		s^{(2,2)}_+(n)-\dfrac{15}{n^5}<	\Biggl(1-\frac{1}{\widehat{\mu}(n+1)}\Biggr)^{\frac{1}{n+1}}<s^{(2,2)}_+(n)+\dfrac{15}{n^5},
		\end{equation}
		where \begin{eqnarray}
			s^{(2,2)}_+(n)&:=&1-\dfrac{1}{\pi n^{3/2}}-\dfrac{1}{2 \pi^2 n^2}+\Bigl(-\frac{1}{3 \pi^3}+\frac{3}{2\pi}\Bigr)\dfrac{1}{n^{5/2}}+\Bigl(-\frac{1}{4 \pi^4}+\frac{3}{2\pi^2}\Bigr)\dfrac{1}{n^{3}}+\Bigl(-\frac{1}{5 \pi^5}+\frac{4}{3\pi^3}\nonumber \\  & &-\frac{15}{8\pi}\Bigr)\dfrac{1}{n^{7/2}} +\Bigl(-\frac{1}{6 \pi^6}+\frac{29}{24\pi^4}-\frac{3}{\pi^2}\Bigr)\dfrac{1}{n^{4}} +\Bigl(-\frac{1}{7 \pi^7} +\frac{67}{60\pi^5}-\frac{27}{8\pi^3}+\frac{35}{16 \pi}\Bigr)\dfrac{1}{n^{9/2}}. \nonumber 
		\end{eqnarray}
		It is clear that $n \geq 2$,
		\begin{equation}\label{eqn23}
		s^{(2,3)}_+(n)-\dfrac{4}{n^5}<	\Biggl(1-\frac{1}{\widehat{\mu}(n)}\Biggr)^{-\frac{1}{n+1}}<s^{(2,3)}_+(n)+\dfrac{4}{n^5},
		\end{equation}
		where \begin{eqnarray}
			s^{(2,3)}_+(n)&:=&1+\dfrac{1}{\pi n^{3/2}}+\dfrac{1}{2 \pi^2 n^2}+\Bigl(\frac{1}{3 \pi^3}-\frac{1}{\pi}\Bigr)\dfrac{1}{n^{5/2}}+\frac{1}{4 \pi^4 n^{3}}+\Bigl(\frac{1}{5 \pi^5}+\frac{1}{6\pi^3}+\frac{1}{\pi}\Bigr)\dfrac{1}{n^{7/2}}\nonumber \\  & & +\Bigl(\frac{1}{6 \pi^6}+\frac{5}{24\pi^4}-\frac{1}{2\pi^2}\Bigr)\dfrac{1}{n^{4}} +\Bigl(\frac{1}{7 \pi^7} +\frac{13}{60\pi^5}-\frac{1}{2\pi^3}-\frac{1}{\pi}\Bigr)\dfrac{1}{n^{9/2}}. \nonumber 
		\end{eqnarray}
		Summerizing \eqref{eqn21}, \eqref{eqn22}, and \eqref{eqn23}, for $n \geq 2$, we have
		\begin{equation}\label{eqn24}
		s^{(2)}_+(n)-\dfrac{40}{n^5}<	\Biggl(\frac{\widehat{\mu}(n+1)-1}{\widehat{\mu}(n)-1}\Biggr)^{\frac{1}{n+1}}<s^{(2)}_+(n)+\dfrac{35}{n^5},
		\end{equation}
		where \begin{eqnarray}
			s^{(2)}_+(n)&:=&1+\dfrac{1}{2 n^2}+\frac{1}{2\pi n^{5/2}}+\Biggl(-\frac{3}{4}+\frac{1}{2 \pi^2}\Biggr)\dfrac{1}{ n^{3}}+\Bigl(\frac{4}{8\pi^3}-\frac{7}{8\pi}\Bigr)\dfrac{1}{n^{7/2}} \nonumber \\  & &+\Bigl(\frac{25}{24}+\frac{1}{2\pi^4}-\frac{1}{\pi^2}\Bigr)\dfrac{1}{n^{4}}+\Bigl(\frac{1}{2\pi^5}-\frac{9}{8\pi^3}+\frac{23}{16\pi}\Bigr)\dfrac{1}{n^{9/2}}.\nonumber
		\end{eqnarray}
		To estimate upper bound of $e^{\frac{\widehat{\mu}(n+1)-\widehat{\mu}(n)}{n+1}}$, first we provide an upper bound of $\frac{\widehat{\mu}(n+1)-\widehat{\mu}(n)}{n+1}$. In a similar framework as done before, we obtain for $n \geq 2$,
		\begin{equation}\label{eqn25}
		s^{(3,1)}_+(n)-\dfrac{2}{n^5}<	\frac{\widehat{\mu}(n+1)-\widehat{\mu}(n)}{n+1}<s^{(3,1)}_+(n)+\dfrac{2}{n^5},
		\end{equation}
		where \begin{equation*}
			s^{(3,1)}_+(n):=\dfrac{\pi}{2n^{3/2}}-\dfrac{5\pi}{8n^{5/2}}+\dfrac{11\pi}{16 n^{7/2}}-\dfrac{93\pi}{128n^{9/2}}.
		\end{equation*}
		It can be easily checked that for $n \geq 2$, $$e^{\frac{2}{ n^5}}< 1+\frac{4}{n^5}$$ and $$e^{-\frac{2}{ n^5}}> 1-\frac{4}{n^5}.$$ 
	Observe that for $n \geq 2$,
		\begin{equation}\label{eqn26}
		s^{(3)}_+(n)-\frac{30}{n^5}<	e^{\frac{\widehat{\mu}(n+1)-\widehat{\mu}(n)}{n+1}}<s^{(3)}_+(n)+\frac{120}{n^5},
		\end{equation}
		where \begin{eqnarray}
			s^{(3)}_+(n)&:=&1+\dfrac{\pi}{2n^{3/2}}-\dfrac{5\pi}{8n^{5/2}}+\frac{\pi^2}{8n^{3}}+\dfrac{11\pi}{16n^{7/2}}-\dfrac{5\pi^2}{16n^4}+\Bigl(-\frac{93\pi}{128}+\frac{\pi^3}{48}\Bigr)\dfrac{1}{n^{9/2}}.\nonumber
		\end{eqnarray}
		Therefore, from \eqref{eqn20}, \eqref{eqn24}, and \eqref{eqn26}, for $n \geq 2$, it gives
		\begin{equation}\label{eqn27}
		UT_-(n)	<\Biggl(\frac{\widehat{T}(n+1)}{\widehat{T}(n)}\Biggr)^{\frac{1}{n+1}}<UT_+(n),
		\end{equation}
		where \begin{equation*}
			UT_+(n):=\Bigl(s^{(1)}_+(n)+\frac{40}{n^5}\Bigr)\Bigl(s^{(2)}_+(n)+\frac{35}{n^5}\Bigr)\Bigl(s^{(3)}_+(n)+\frac{120}{n^5}\Bigr)	
		\end{equation*}
	and \begin{equation*}
		\hspace*{-0.2cm}UT_-(n):=\Bigl(s^{(1)}_+(n)-\frac{20}{n^5}\Bigr)\Bigl(s^{(2)}_+(n)-\frac{40}{n^5}\Bigr)\Bigl(s^{(3)}_+(n)-\frac{30}{n^5}\Bigr).	
	\end{equation*}
		For the remaining part in the upper bound of \eqref{eqn20*}, we get for $n \geq 2$,
		\begin{equation}\label{eqn28}
			\Biggl(1+\dfrac{6\cdot2^{10}} {\widehat{\mu}(n)^{10}}\Biggr)^{\frac{1}{n+1}}<1+\dfrac{4}{n^5},
		\end{equation} and \begin{equation}\label{eqn28'}
		\Biggl(1-\dfrac{4\cdot2^{10}} {\widehat{\mu}(n)^{10}}\Biggr)^{\frac{1}{n+1}}>1-\dfrac{2}{n^5}.
	\end{equation}
 Equations \eqref{eqn27}, \eqref{eqn28} and \eqref{eqn28'} together imply the upper bound of $\sqrt[n+1]{\frac{\overline{p}(n+1)}{\overline{p}(n)}}$ for all $n \geq 1297$, stated in \eqref{lem3eqn31}. For the remaining cases $1\leq n \leq 1296$, one can numerically verify \eqref{lem3eqn31} with Mathematica.
		\end{proof}
\begin{lemma}\label{lem4}
	For $n \geq 2$, 
	\begin{equation}\label{lem3eqn3}
		s_-(n)-\dfrac{110}{n^5}<\sqrt[n-1]{\dfrac{\overline{p}(n-1)}{\overline{p}(n)}}<s_-(n)+\dfrac{200}{n^5}
	\end{equation}
	where $s_-(n):=\displaystyle\sum_{i=0}^{9}\dfrac{b_i}{\sqrt{n}^i}$ with $b_0=1$ and
	\begin{equation*}
	\begin{split}
	&b_1=b_2=0, b_3=-\dfrac{\pi}{2}, b_4=1, b_5=-\dfrac{4+5\pi^2}{8\pi}, b_6=-\dfrac{4+12\pi^2+\pi^4}{8\pi^2}, b_7=-\dfrac{8+14\pi^2+19\pi^4}{16\pi^3},\\
	&b_8=-\dfrac{24+48\pi^2-124\pi^4-15\pi^6}{48\pi^4}, b_9=-\dfrac{192+432\pi^2+552\pi^4+807\pi^6+8\pi^8}{384\pi^5}.
	\end{split}
	\end{equation*}
\end{lemma}
\begin{proof}
Analogous to the proof of Lemma \ref{lem3}.
\end{proof}
\begin{lemma}\label{lemnew1}
Let $\widetilde{N}(\alpha)$ be defined as in \eqref{defnew12}. Then for $n \geq \widetilde{N}(\alpha)$,
\begin{equation}\label{lemnew1eqn1}
1+\frac{3 \alpha-2 \alpha \log n}{n^3}-\frac{2^{7 \alpha+14}\log n}{n^5}<\dfrac{n^\frac{2 \alpha}{n}}{(n-1)^{\frac{\alpha}{n-1}}(n+1)^{\frac{\alpha}{n+1}}}<1+\frac{3 \alpha-2 \alpha \log n}{n^3}+\frac{2^{18 \alpha+31}}{n^5}.
\end{equation} 
\end{lemma}
\begin{proof}
Splitting the quotient given on the left hand side of \eqref{lemnew1eqn1} as follows
\begin{equation}\label{eqna1}
	\dfrac{n^\frac{2 \alpha}{n}}{(n-1)^{\frac{\alpha}{n-1}}(n+1)^{\frac{\alpha}{n+1}}}= \underset{:=X_{1,\alpha}(n)}{\underbrace{\Bigl(1-\frac{1}{n}\Bigr)^{-\frac{\alpha}{n-1}}}} \underset{:=X_{2,\alpha}(n)}{\underbrace{\Bigl(1+\frac{1}{n}\Bigr)^{-\frac{\alpha}{n+1}}}} \underset{:=X_{3,\alpha}(n)}{\underbrace{n^{\frac{2\alpha}{n}-\frac{\alpha}{n+1}-\frac{\alpha}{n-1}}}}.
\end{equation}
First taking the logarithm of $X_{1,\alpha}(n)$ and by Taylor expansion, we have for all $n \geq 3$, 
\begin{equation*}
{\alpha}\Biggl(\frac{1}{n^2}+\frac{3}{2n^3}+\frac{11}{6n^4}\Biggr)<\log \Bigl(X_{1,\alpha}(n)\Bigr)<{\alpha}\Biggl(\frac{1}{n^2}+\frac{3}{2n^3}+\frac{11}{6n^4}+\frac{5}{n^5}\Biggr), 
\end{equation*}
and therefore,
 \begin{equation}\label{eqna2}
e^{{\alpha}\Bigl(\frac{1}{n^2}+\frac{3}{2n^3}+\frac{11}{6n^4}\Bigr)}<	X_{1,\alpha}(n)<e^{{\alpha}\Bigl(\frac{1}{n^2}+\frac{3}{2n^3}+\frac{11}{6n^4}+\frac{5}{n^5}\Bigr)}.
\end{equation}
By Taylor expansion of $e^{\frac{\alpha}{n^2}}$, for all $n \geq \lceil\sqrt{2\alpha}\rceil$,
\begin{equation}\label{eqna3}
	1+\frac{\alpha}{n^2}+\frac{\alpha^2}{2n^4}<e^{\frac{\alpha}{n^2}}<1+\frac{\alpha}{n^2}+\frac{\alpha^2}{2n^4}+\frac{2 \alpha^3}{n^5},
\end{equation}
since $\displaystyle\sum_{k=3}^{\infty} \frac{1}{k!}\Bigl(\frac{\alpha}{n^2}\Bigr)^{k}<\frac{\alpha^3}{n^5}\Bigl(1-\frac{\alpha}{n^2}\Bigr)^{-1}<\frac{2\alpha^3}{n^5}$ for all $n\geq \lceil \sqrt{2\alpha} \rceil$. 
Similarly for the remaining exponential factors of \eqref{eqna2}, for all $n\geq \max\Bigl\{\Bigl\lceil\sqrt[3]{3\alpha}\Bigr\rceil,\Bigl\lceil\sqrt[4]{11\alpha/3}\Bigr\rceil,\Bigl\lceil\sqrt[5]{10\alpha}\Bigr\rceil\Bigr\}$,
\begin{equation}\label{eqna4}
1+\frac{3\alpha}{2n^3}<e^{\frac{3\alpha}{2n^3}}<1+\frac{3\alpha}{2n^3}+\frac{9\alpha^2}{2n^5}, 	1+\frac{11\alpha}{6n^4}<e^{\frac{11\alpha}{6n^4}}<1+\frac{11\alpha}{6n^4}+\frac{121\alpha^2}{18n^5},\ \text{and}\ \ e^{\frac{5\alpha}{n^5}}<1+\frac{10\alpha}{n^5}.
\end{equation}
Using the fact that $\alpha < 2^{\alpha}$ for all $\alpha \in {\mathbb{R}}_{\geq 0}$, from \eqref{eqna3} and \eqref{eqna4}, we have for all $n \geq \widetilde{N_1}(\alpha)$ (cf. see \eqref{defnew9}),
\begin{eqnarray}
	X_{1,\alpha}(n) &<&	\left(1+\frac{\alpha^2}{n^2}+\frac{\alpha^2}{2 n^4}+\frac{2 \alpha^3}{n^5}\right) \left(1+\frac{3 \alpha}{2 n^3}+\frac{9 \alpha^2}{2 n^5}\right) \left(1+\frac{11 \alpha}{6 n^4}+\frac{121 \alpha^2}{18 n^5}\right) \left(1+\frac{10 \alpha}{n^5}\right) \nonumber \\ \label{eqna7}&<& 1+\frac{\alpha}{n^2}+\frac{3\alpha}{2n^3}+\frac{ \alpha(11+3 \alpha)}{6n^4}+\frac{2^{12+8 \alpha}}{n^5},
\end{eqnarray} and for all $n \geq 3$, \begin{eqnarray}
X_{1,\alpha}(n) &>&	\left(1+\frac{\alpha^2}{n^2}+\frac{\alpha^2}{2 n^4}\right) \left(1+\frac{3 \alpha}{2 n^3}\right) \left(1+\frac{11 \alpha}{6 n^4}\right) \nonumber \\ \label{eqna7'}&>& 1+\frac{\alpha}{n^2}+\frac{3\alpha}{2n^3}+\frac{ \alpha(11+3 \alpha)}{6n^4}.
\end{eqnarray}
Similarly, for all $n \geq \widetilde{N_2}(\alpha)$ (cf. see \eqref{defnew10}), we have
\begin{equation}\label{eqna8}
1-\frac{\alpha}{n^2}+\frac{3\alpha}{2n^3}+\frac{ \alpha(-11+3 \alpha)}{6n^4}-\frac{19\cdot 2^{5\alpha+2}}{n^5}<	X_{2,\alpha}(n)<1-\frac{\alpha}{n^2}+\frac{3\alpha}{2n^3}+\frac{ \alpha(-11+3 \alpha)}{6n^4}+\frac{31\cdot 2^{7\alpha+4}}{n^5}.
\end{equation}
For the remaining part $X_{3,\alpha}(n)$ in \eqref{eqna1}, observe that for all $n \geq 2$,
\begin{equation*}
\Bigl(-\frac{2 \alpha}{n^3}-\frac{4 \alpha}{n^5}\Bigr) \log n<	\log \Bigl(X_{3,\alpha}(n)\Bigr) < \Bigl(-\frac{2 \alpha}{n^3}-\frac{2 \alpha}{n^5}\Bigr) \log n,\end{equation*} which implies \begin{equation*} e^{\bigl(-\frac{2 \alpha}{n^3}-\frac{4 \alpha}{n^5}\bigr)\log n}<X_{3,\alpha}(n)<e^{\bigl(-\frac{2 \alpha}{n^3}-\frac{2 \alpha}{n^5}\bigr)\log n}.
\end{equation*}
Consider the following series expansion of the exponential function:
\begin{equation*}
	e^{\bigl(-\frac{2 \alpha}{n^3}-\frac{2 \alpha}{n^5}\bigr) \log n }=1-\frac{2 \alpha \log n}{n^3}-\frac{2 \alpha \log n}{n^5}+\Bigl(\frac{2 \alpha \log n}{n^3}\Bigr)^{2} \frac{(1+\frac{1}{n^2})^2}{2!} +\sum_{k=3}^{\infty} (-1)^k \Bigl(\frac{2 \alpha \log n}{n^3}\Bigr)^{k} \frac{(1+\frac{1}{n^2})^k}{k!}.
\end{equation*}
Since for all $n \geq \max \Bigl\{5505,\Bigl\lceil \Bigl(\frac{4 \alpha}{\alpha^2+1}\Bigr)^{4/3} \Bigr\rceil\Bigr\}$, \begin{equation*}
	-\frac{2 \alpha \log n}{n^5}+\Bigl(\frac{2 \alpha \log n}{n^3}\Bigr)^{2} \frac{(1+\frac{1}{n^2})^2}{2!} < \frac{\alpha^3}{n^5},
\end{equation*}
and
\begin{eqnarray}
	\sum_{k=3}^{\infty} (-1)^k \Bigl(\frac{2 \alpha \log n}{n^3}\Bigr)^{k} \frac{(1+\frac{1}{n^2})^k}{k!} &<& \sum_{k=3}^{\infty} \Bigl(\frac{2 \alpha}{n^{\frac{11}{4}}}\Bigr)^{k} \Bigl(\text{since}\ \  \frac{(1+\frac{1}{n^2})^k}{k!}<1\ \ \text{for all}\ \  n\geq 2, k\geq 2\Bigr)\nonumber \\ &<& \frac{16 \alpha^3}{n^5}\ \ \text{for all}\ \  n\geq \max\Bigl\{5505, \lceil (4 \alpha)^{4/11} \rceil\Bigr\}.\nonumber
\end{eqnarray}
Therefore, combining the above estimations, for all $n \geq \widetilde{N_3}(\alpha)$ (cf. see \eqref{defnew11}), we have 
\begin{equation}\label{eqna9}
	X_{3,\alpha}(n)<1-\frac{2 \alpha \log n}{n^3}+\frac{17 \alpha^3}{n^5}.
\end{equation}
Analogous to the above we have for all $n \geq \max \Bigl\{5505,\lceil (4 \alpha)^{4/11} \rceil\Bigr\}$,
\begin{equation*}
	\mid e^{\alpha (-\frac{2}{n^3}-\frac{4}{n^5})\log n}-\Bigl(1-\frac{2 \alpha}{n^3} \Bigl(1+\frac{2}{n^2}\Bigr)\log n\Bigr)|=\mid \sum_{k=2}^{\infty} \frac{1}{k!} \Bigl(\frac{2 \alpha \log n}{n^3} \Bigl(1+\frac{2}{n^2}\Bigr)\Bigr)^k \mid <\frac{24 {\alpha}^2}{n^5}.
\end{equation*} Thus for all $n \geq \widetilde{N_3}(\alpha)$, \begin{equation}\label{eqna9'}
X_{3,\alpha}(n)>1-\frac{2 \alpha \log n}{n^3}-\frac{4 \alpha \log n+24 \alpha^2}{n^5}.
\end{equation}
Applying \eqref{eqna7}-\eqref{eqna9'}, for all $n \geq \underset{1\leq i \leq 3}{\max}\Bigl\{\widetilde{N_i}(\alpha)\Bigr\}=\widetilde{N}(\alpha)$, it follows that 
\begin{eqnarray}
	\dfrac{n^\frac{2 \alpha}{n}}{(n-1)^{\frac{\alpha}{n-1}}(n+1)^{\frac{\alpha}{n+1}}}&<&  \Biggl(1+\frac{\alpha}{n^2}+\frac{3\alpha}{2n^3}+\frac{ \alpha(11+3 \alpha)}{6n^4}+\frac{2^{12+8\alpha}}{n^5}\Biggr) \Biggl(1-\frac{\alpha}{n^2}+\frac{3\alpha}{2n^3} \nonumber \\ & & +\frac{ \alpha(-11+3 \alpha)}{6n^4}\ +\frac{31\cdot 2^{7\alpha+4}}{n^5}\Biggr) \Biggl(1-\frac{2 \alpha \log n}{n^3}+\frac{17 \alpha^3}{n^5}\Biggr) \nonumber \\ &<& 1+\frac{3 \alpha-2 \alpha \log n}{n^3}+ \frac{2^{18 \alpha+30}}{n^5}+\frac{2^{10 \alpha+17}}{n^5} \nonumber \\ &<& 1+\frac{3 \alpha-2 \alpha \log n}{n^3}+\frac{2^{18 \alpha+31}}{n^5}, \nonumber
\end{eqnarray} and \begin{eqnarray}
\dfrac{n^\frac{2 \alpha}{n}}{(n-1)^{\frac{\alpha}{n-1}}(n+1)^{\frac{\alpha}{n+1}}}&>& 1+\frac{3 \alpha-2 \alpha \log n}{n^3}+ A(n,\alpha)+B(n,\alpha) \nonumber \\ &>&1+\frac{3 \alpha-2 \alpha \log n}{n^3} -\frac{2^{7 a+12}}{n^5}-\frac{2^{5 a+5} \log n}{n^5}   \nonumber \\ &>& 1+\frac{3 \alpha-2 \alpha \log n}{n^3}-\frac{2^{7 \alpha+14}\log n}{n^5}  \Bigl(\text{since}\ \  -(1+\log n)>-4 \log n\ \ \nonumber \\& & \hspace*{6cm}\text{for all}\ \  n\geq 2\Bigr), \nonumber
\end{eqnarray}
where  \begin{eqnarray}
\hspace*{-0.7cm}	A(n,\alpha)&:=&-\frac{4 \left(6 \alpha^2+19\ 2^{5 a}\right)}{n^5}-\frac{17 \alpha^2}{12 n^6}-\frac{\alpha \left(19\ 2^{5 \alpha+3}\right)}{2 n^7}-\frac{\alpha \left(2592 \alpha^2+121 \alpha+513\ 2^{5 \alpha+3}\right)}{36 n^8} \nonumber \\  & &-\frac{\alpha \left(57\ 2^{5 \alpha+2} \alpha+209\ 2^{5 \alpha+2}\right)}{6 n^9}-\frac{108 \alpha^5}{3 n^{12}}-\frac{6 \alpha^3}{9 n^{13}}, \nonumber
\end{eqnarray}
 \begin{eqnarray}
 	B(n,\alpha)&:=& -\frac{4 \alpha \log n}{n^5}-\frac{6 \alpha^2 \alpha \log n}{n^6}-\frac{12 \alpha^2 \log n}{n^8}-\frac{3 \alpha^4 \log n}{n^{10}}-\frac{ \alpha^5 \log n}{2n^{11}}-\frac{6 \alpha^4 \log n}{n^{12}}-\frac{ \alpha^2 \log n}{n^{13}}. \nonumber
 \end{eqnarray}
\end{proof}
	\begin{lemma}\label{lemnew2} For all $n \geq 2$,
		$$L^{+}(n)<\overline{p}(n)^{\frac{2}{n(n-1)(n+1)}}<U^{+}(n),$$ where
		$$\hspace*{-0.4cm}U^{+}(n):= 1+\frac{2 \pi}{n^{5/2}}-\frac{2 \log (8n)}{n^3}-\frac{2}{\pi n^{7/2}}-\frac{1}{\pi^2 n^4}+\frac{6 \pi^4-2}{3 \pi^3 n^{9/2}}+\frac{140}{n^5},$$ and $$L^{+}(n):=1+\frac{2 \pi}{n^{5/2}}-\frac{2 \log (8n)}{n^3}-\frac{2}{\pi n^{7/2}}-\frac{1}{\pi^2 n^4}+\frac{6 \pi^4-2}{3 \pi^3 n^{9/2}}-\frac{80 \log n}{n^5}.$$
	\end{lemma}
	\begin{proof}
		From \cite[Lemma 2.1]{M2}, for all $n \geq 94$, we have 
		\begin{equation}\label{lem4eqn1}
			\frac{e^{\widehat{\mu}(n)}}{8n} \Biggl(1-\frac{1}{\widehat{\mu}(n)}-\frac{1}{\widehat{\mu}(n)^6}\Biggr)<\overline{p}(n)<\frac{e^{\widehat{\mu}(n)}}{8n} \Biggl(1-\frac{1}{\widehat{\mu}(n)}+\frac{1}{\widehat{\mu}(n)^6}\Biggr).
		\end{equation}
		In order to give an upper bound of ${\overline{p}(n)^{\frac{2}{n(n-1)(n+1)}}}$, we shall estimate an upper bound for each of the three factors, namely, $e^{\frac{2 \widehat{\mu}(n)}{n(n-1)(n+1)}}$, $(8n)^{-\frac{2}{n(n-1)(n+1)}}$, and $\Bigl(1-\frac{1}{\widehat{\mu}(n)}+\frac{1}{\widehat{\mu}(n)^6}\Bigr)^{\frac{2}{n(n-1)(n+1)}}$.\\
	Note that for all $n \geq 2$,
		\begin{equation*}
		\frac{2 \pi}{n^{5/2}}+\frac{2 \pi}{n^{9/2}}<	\frac{2 \widehat{\mu}(n)}{n(n-1)(n+1)}< \frac{2 \pi}{n^{5/2}}+\frac{2 \pi}{n^{9/2}}+\frac{4 \pi}{n^5},
		\end{equation*}
		which implies that $n \geq 3$,
		\begin{equation}\label{lem4eqn2}
	1+\frac{2 \pi}{n^{5/2}}+\frac{2 \pi}{n^{9/2}}	<	e^{\frac{2 \widehat{\mu}(n)}{n(n-1)(n+1)}}<e^{\frac{2 \pi}{n^{5/2}}+\frac{2 \pi}{n^{9/2}}+\frac{4 \pi}{n^5}}<1+\frac{2 \pi}{n^{5/2}}+\frac{2 \pi}{n^{9/2}}+\frac{125}{n^5}.
		\end{equation}
	Taking logarithm of  $\Bigl(1-\frac{1}{\widehat{\mu}(n)}+\frac{1}{\widehat{\mu}(n)^6}\Bigr)^{\frac{2}{n(n-1)(n+1)}}$ and $\Bigl(1-\frac{1}{\widehat{\mu}(n)}-\frac{1}{\widehat{\mu}(n)^6}\Bigr)^{\frac{2}{n(n-1)(n+1)}}$, and noting that for all $n \geq 2$,
	\begin{equation*}
	\frac{2}{n(n-1)(n+1)} \log \Bigl(1-\frac{1}{\widehat{\mu}(n)}+\frac{1}{\widehat{\mu}(n)^6}\Bigr) < -\frac{2}{\pi n^{7/2}}-\frac{1}{\pi^2 n^4}-\frac{2}{3 \pi^3 n^{9/2}}, 
	\end{equation*}
\begin{equation*}
	\frac{2}{n(n-1)(n+1)} \log \Bigl(1-\frac{1}{\widehat{\mu}(n)}-\frac{1}{\widehat{\mu}(n)^6}\Bigr) > -\frac{2}{\pi n^{7/2}}-\frac{1}{\pi^2 n^4}-\frac{2}{3 \pi^3 n^{9/2}}-\frac{2}{n^5},
\end{equation*} 
	which at once implies that for all $n \geq 2$,
		\begin{equation}\label{lem4eqn3}
			\Bigl(1-\frac{1}{\widehat{\mu}(n)}+\frac{1}{\widehat{\mu}(n)^6}\Bigr)^{\frac{2}{n(n-1)(n+1)}}< 1-\frac{2}{\pi n^{\frac{7}{2}}}-\frac{1}{\pi^2 n^4}-\frac{2}{3 \pi^3 n^{\frac{9}{2}}}+\frac{1}{n^5},
		\end{equation}
	and \begin{equation}\label{lem4eqn3'}
		\Bigl(1-\frac{1}{\widehat{\mu}(n)}-\frac{1}{\widehat{\mu}(n)^6}\Bigr)^{\frac{2}{n(n-1)(n+1)}}> 1-\frac{2}{\pi n^{\frac{7}{2}}}-\frac{1}{\pi^2 n^4}-\frac{2}{3 \pi^3 n^{\frac{9}{2}}}-\frac{5}{n^5}.
	\end{equation}
 Similarly, we obtain for all $n \geq 688$,
		\begin{equation}\label{lem4eqn5}
	1-\frac{2 \log (8n)}{n^3}-\frac{4 \log 8n}{n^5}	<(8n)^{-\frac{2}{n(n-1)(n+1)}}	<1-\frac{2 \log (8n)}{n^3}.
		\end{equation}
	Combining \eqref{lem4eqn2}-\eqref{lem4eqn5}, it is immediate that for all $n \geq 688$,
		\begin{equation}\label{lem4eqn6}
		\hspace*{-0.215cm}	\overline{p}(n)^{\frac{2}{n(n-1)(n+1)}}<1+ \frac{2 \pi}{n^{\frac{5}{2}}}-\frac{2 \log (8n)}{n^3}-\frac{2}{\pi n^{\frac{7}{2}}}-\frac{1}{\pi^2 n^4}+\frac{6 \pi^4-2}{3 \pi^3 n^{\frac{9}{2}}}+\frac{140}{n^5},
		\end{equation}
	\begin{equation}\label{lem4eqn6'}
		\hspace*{0.3cm}\overline{p}(n)^{\frac{2}{n(n-1)(n+1)}}>1+ \frac{2 \pi}{n^{\frac{5}{2}}}-\frac{2 \log (8n)}{n^3}-\frac{2}{\pi n^{\frac{7}{2}}}-\frac{1}{\pi^2 n^4}+\frac{6 \pi^4-2}{3 \pi^3 n^{\frac{9}{2}}}-\frac{80 \log n}{n^5},
	\end{equation}
	To conclude the proof, it remains to check \eqref{lem4eqn6'} for $2\leq n \leq 687$ which we did numerically with Mathematica.
	\end{proof}

	\begin{theorem}\label{newtheorem}
	Let $\widetilde{N}(\alpha)$ be as in \eqref{defnew12}. Then for all $n \geq \widetilde{N}(\alpha)$,
	\begin{equation*}
	L(n,\alpha)<u_{\alpha}(n)<U(n,\alpha)\ \ \text{and}\ \ L_1(n,\alpha)<u_{\alpha}(n+1) < U_{1}(n,\alpha),
	\end{equation*}
		where \begin{eqnarray}
		\hspace*{0cm}	U(n,\alpha) &:=&1+\frac{3 \pi}{4 n^{5/2}}+\frac{3+3 \alpha-2 \log 8-2(1+\alpha)\log n}{n^3}-\frac{15}{4 \pi n^{7/2}}-\frac{3}{\pi^2 n^4}+\frac{35(-16+3 \pi^4)}{192 \pi^3 n^{9/2}}\nonumber \\ & & +\frac{6 \alpha+2^{12}+2^{31+18 \alpha}+(2+17 \alpha)\log n}{n^5}, \nonumber
		\end{eqnarray}
\begin{eqnarray}
	\hspace*{0cm}	L(n,\alpha) &:=&1+\frac{3 \pi}{4 n^{5/2}}+\frac{3+3 \alpha-2 \log 8-2(1+\alpha)\log n}{n^3}-\frac{15}{4 \pi n^{7/2}}-\frac{3}{\pi^2 n^4}+\frac{35(-16+3 \pi^4)}{192 \pi^3 n^{9/2}}\nonumber \\ & & -\frac{258+(80+2^{16+7 \alpha}) \log n}{n^5}, \nonumber
\end{eqnarray}
	\begin{eqnarray}
	\hspace*{0cm}U_{1}(n,\alpha)&:=&1+\frac{3 \pi}{4 n^{5/2}}+\frac{3+3 \alpha-2 \log 8-2(1+\alpha)\log n}{n^3}-\frac{15(2+\pi^2)}{8 \pi n^{7/2}}+\frac{1}{\pi^2 n^4}\Bigl(-3-11\pi^2\nonumber \\ & &-11\alpha \pi^2 +6 \pi^2 \log 8+(6+6 \alpha)\pi^2 \log n\Bigr)+\frac{35(21 \pi^4+72 \pi^2-16)}{192 \pi^3 n^{9/2}} +\frac{1}{n^5}\Bigl(12\nonumber \\ & &+ (2^{12}+26)\pi^2 -12 \pi^2 \log 8+32 \alpha \pi^2+\pi^2 2^{31+18 \alpha}+(2+17 \alpha) \pi^2 \log n\Bigr),\nonumber
\end{eqnarray}
\begin{eqnarray}
	\hspace*{0cm}L_{1}(n,\alpha)&:=&1+\frac{3 \pi}{4 n^{5/2}}+\frac{3+3 \alpha-2 \log 8-2(1+\alpha)\log n}{n^3}-\frac{15(2+\pi^2)}{8 \pi n^{7/2}}+\frac{1}{\pi^2 n^4}\Bigl(-3-11\pi^2\nonumber \\ & &-11\alpha \pi^2 +6 \pi^2 \log 8+(6+6 \alpha)\pi^2 \log n\Bigr)+\frac{35(21 \pi^4+72 \pi^2-16)}{192 \pi^3 n^{9/2}} -\frac{1}{n^5}\Bigl(256\nonumber \\ & &+ \frac{105 \left(9 \pi ^2 \left(4+\pi ^2\right)-16\right)}{128 \pi ^3}+8 \left(3 \alpha+2^{7 \alpha+13}+13\right) \log n \Bigr).\nonumber
\end{eqnarray}
	\end{theorem}
\begin{proof}
Applying Lemmas \ref{lem3}-\ref{lemnew2}, can deduce that for all $n\geq \widetilde{N}(\alpha)$, $L(n,\alpha)<u_{\alpha}(n)<U(n,\alpha)$, and applying this inequality, further it can be shown that $L_1(n,\alpha)<L(n+1,\alpha)<u_{\alpha}(n+1)<U(n+1,\alpha)<U_1(n,\alpha)$ utilizing the following two inequalities
\begin{equation*}
-\frac{\log n}{n^3}+\frac{3 \log n-1}{n^4}+\frac{1-12 \log n}{n^5}<	-\frac{\log (n+1)}{(n+1)^3}<-\frac{\log n}{n^3}+\frac{3 \log n-1}{n^4}+\frac{4}{n^5}
\end{equation*} and 
\begin{equation*}
	-\frac{\log n}{n^5}<\frac{\log (n+1)}{(n+1)^5}<\frac{\log n}{n^5}.
\end{equation*} Due to its routine verification, we omit the proof here.
\end{proof}
Before we move on to prove Theorem \ref{mainresult3}, we refine the upper bounds (resp. lower bounds) $U(n,\alpha)$ and $U_1(n,\alpha)$ (resp. for $L(n,\alpha)$ and $L_1(n,\alpha)$). First, we need to refine the constants depending on $n$ that are coefficients of $\frac{1}{n^5}$ in $U(n,\alpha), U_1(n,\alpha), L(n,\alpha)$ and $L_1(n,\alpha)$ by using the following estimation.
\begin{lemma}\label{newlemma1}
For $(a, b)\in \mathbb{R}\times \mathbb{R}_{>1}$, for all $n\ge \max \left\{\left\lceil e^{\frac{a}{b}}\right\rceil, \left\lceil (2b)^{19}\right\rceil \right\}$,
$$a+b\log(n)<n^{\frac{1}{4}}.$$	
\end{lemma}
\begin{proof}
To begin with, we see that $a+b\log (n)=b\log (n)\left(1+\frac{a}{b\log (n)}\right)$ and for all $n\ge \left\lceil e^{\frac{a}{b}}\right\rceil$, $\frac{a}{b\log (n)}<1$ and therefore, it follows that $a+b\log (n)<2b\log(n)$. Now it remains to show that $2b\log(n)<n^{\frac 14}$ which is equivalent to $4\log(2b)<\log(n)-4\log\log(n)$. Define $f(x):=\log(x)-4\log\log(x)$ and observe that for all $x\ge e^4$, $f(x)$ is increasing and so is $f(n)$ for $n\ge \lceil e^4\rceil$. What remains to prove $4\log(2b)<\log(n)-4\log\log(n)$ is to take an appropriate choice of $n=n_0(>e^4)$ so that $f(n_0)>4\log(2b)$ and since $f$ is increasing, we can conclude that $f(n)>4\log(2b)$. So, choosing $n_0=(2b)^{19}$ (and note that $n_0>e^4$ as $b>1$), we now have to prove
$$f(n_0)=19\log(2b)-4\log(19\log(2b))>4\log(2b)\iff 15\log(2b)>4\log(19\log(2b)).$$
Setting $g(y):=15y-4\log (19y)$, it is clear that $g(y)$ is increasing for $y>\frac{4}{15}$ and taking $y_0=\log(2)$, we have $g(\log(2))>0$. So for all $y>y_0$, $g(y)>0$; in particular choose $y=\log(2b)$ to conclude that $g(\log(2b))>0$. Hence, for all $n\ge \left\lceil (2b)^{19}\right\rceil$, $f(n)\ge f(n_0)>4\log (2b)$. This concludes the proof. 
\end{proof}
Applying Lemma \ref{newlemma1} with $(a,b)=\left(6\alpha+2^{12}+2^{31+18\alpha},2+17\alpha\right)$, for all
$$n\ge \max\left\{ \left\lceil e^{\frac{6\alpha+2^{12}+2^{31+18\alpha}}{2+17\alpha}}\right\rceil, \left\lceil (4+34\alpha)^{19}\right\rceil   \right\}=:n_u(\alpha),$$
we have
 $$\frac{6 \alpha+2^{12}+2^{31+18 \alpha}+(2+17 \alpha)\log n}{n^5}<\frac{1}{n^{\frac{19}{4}}}.$$
Consequently, it follows that for $n\ge n_u(\alpha)$,
\begin{equation}\label{Gneweqn1}
U(n,\alpha)<1+\frac{1}{n^{\frac 52}}\sum_{k=0}^{4}\frac{c_k^{[1]}(\alpha,n)}{\sqrt{n}^k}+\frac{1}{n^{\frac{19}{4}}}:=U^*(n,\alpha),
\end{equation}
with
\begin{align}\nonumber
&c_0^{[1]}(\alpha,n)=\frac{3\pi}{4},\  c_1^{[1]}(\alpha,n)=3+3\alpha-2\log 8-2(1+\alpha)\log n,\  c_2^{[1]}(\alpha,n)=-\frac{15}{4\pi},\\\label{Gnewcoeff1}  &c_3^{[1]}(\alpha,n)=-\frac{3}{\pi^2},\ c_4^{[1]}(\alpha,n)=\frac{35(-16+3\pi^4)}{192\pi^3}.
\end{align}
Similarly, applying Lemma \ref{newlemma1} with $(a,b)=\left(258,80+2^{16+7\alpha}\right)$, we have for all
$$n\ge \max\left\{ \left\lceil e^{\frac{258}{80+2^{16+7\alpha}}}\right\rceil, \left\lceil (160+2^{17+7\alpha})^{19}\right\rceil   \right\}=:n_l(\alpha),\  \frac{258+(80+2^{16+7\alpha})\log n}{n^5}<\frac{1}{n^{\frac{19}{4}}}.$$
Consequently, it follows that for $n\ge n_l(\alpha)$,
\begin{equation}\label{Gneweqn2}
L(n,\alpha)>1+\frac{1}{n^{\frac 52}}\sum_{k=0}^{4}\frac{c_k^{[1]}(\alpha,n)}{\sqrt{n}^k}-\frac{1}{n^{\frac{19}{4}}}:=L^*(n,\alpha).
\end{equation}
Next, applying Lemma \ref{newlemma1} with $$(a,b)=\left(12+4122\pi^2-12\pi^2\log 8+32\alpha\pi^2+2^{31+18\alpha}\pi^2,(2+17\alpha)\pi^2\right),$$ for all 
$$n\ge \max\left\{ \left\lceil e^{\frac{12+4122\pi^2-12\pi^2\log 8+32\alpha\pi^2+2^{31+18\alpha}\pi^2}{\pi^2(2+17\alpha)}}\right\rceil, \left\lceil \pi^{38}(4+34\alpha)^{19}\right\rceil   \right\}=:n_{u_1}(\alpha),$$
it follows that
$$\frac{12+4122\pi^2-12\pi^2\log 8+32\alpha\pi^2+2^{31+18\alpha}\pi^2+\pi^2(2+17 \alpha)\log n}{n^5}<\frac{1}{n^{\frac{19}{4}}}.$$
Consequently, for $n\ge n_{u_1}(\alpha)$,
\begin{equation}\label{Gneweqn3}
U_1(n,\alpha)<1+\frac{1}{n^{\frac 52}}\sum_{k=0}^{4}\frac{c_k^{[2]}(\alpha,n)}{\sqrt{n}^k}+\frac{1}{n^{\frac{19}{4}}}:=U_1^*(n,\alpha),
\end{equation}
with
\begin{align}\nonumber
&c_0^{[2]}(\alpha,n)=c_0^{[1]}(\alpha,n),\  c_1^{[2]}(\alpha,n)=c_1^{[1]}(\alpha,n),\  c_2^{[2]}(\alpha,n)=-\frac{15(2+\pi^2)}{8\pi},\\\label{Gnewcoeff2} &c_3^{[2]}(\alpha,n)=-\frac{3+11\pi^2+11\alpha\pi^2-6\pi^2\log 8-(6+6\alpha)\pi^2\log n}{\pi^2},\ c_4^{[2]}(\alpha,n)=\frac{35(21\pi^4+72\pi^2-16)}{192\pi^3}.
\end{align}
Finally, applying Lemma \ref{newlemma1} with $(a,b)=\left(256+\frac{105 \left(9 \pi ^2 \left(4+\pi ^2\right)-16\right)}{128 \pi ^3}, 8(3\alpha+2^{13+7\alpha}+13)\right)$, for all
$$n\ge \max\left\{ \left\lceil e^{\frac{256+\frac{105 \left(9 \pi ^2 \left(4+\pi ^2\right)-16\right)}{128 \pi ^3}}{8(3\alpha+2^{13+7\alpha}+13)}}\right\rceil, \left\lceil 16^{19}(3\alpha+2^{13+7\alpha}+13)^{19}\right\rceil   \right\}=:n_{l_1}(\alpha),$$
it follows that
$$ \frac{256+\frac{105 \left(9 \pi ^2 \left(4+\pi ^2\right)-16\right)}{128 \pi ^3}+8(3\alpha+2^{13+7\alpha}+13)\log n}{n^5}<\frac{1}{n^{\frac{19}{4}}}.$$
Consequently, for $n\ge n_{l_1}(\alpha)$,
\begin{equation}\label{Gneweqn4}
L(n,\alpha)>1+\frac{1}{n^{\frac 52}}\sum_{k=0}^{4}\frac{c_k^{[2]}(\alpha,n)}{\sqrt{n}^k}-\frac{1}{n^{\frac{19}{4}}}:=L_1^*(n,\alpha).
\end{equation}
Define
\begin{equation}\label{Gnewdef}
N_1(\alpha):=\max\{n_u(\alpha), n_l(\alpha), n_{u_1}(\alpha), n_{l_1}(\alpha)\}.
\end{equation}
Now, we compute upper bounds for $\underset{0\le k\le 4}{\max}\{|c_k^{[1]}(\alpha,n)|\}$ and $\underset{0\le k\le 4}{\max}\{|c_k^{[2]}(\alpha,n)|\}$ which will be used later in Subsection \ref{prof2}. The following lemma is a prerequisite for the estimations to be done in this context.
\begin{lemma}\label{newlemma2}
	For $(a, b)\in \mathbb{R}\times \mathbb{R}_{>1}$, for all $n\ge \max \left\{\left\lceil e^{\frac{a}{b}}\right\rceil, \left\lceil (2b)^{227}\right\rceil \right\}$,
	$$a+b\log(n)<n^{\frac{1}{16}}.$$	
\end{lemma}
\begin{proof}
We omit the proof due to its similarity with the proof of Lemma \ref{newlemma1}.
\end{proof}

Following \eqref{Gnewcoeff1}, by numercial computations with Mathematica, it follows that
\begin{equation}\label{Gnewcoeff1estim1}
\underset{0\le k\ne 1\le 4}{\max}\{|c_k^{[1]}(\alpha,n)|\}<2.4,
\end{equation}
and for $k=1$,
\begin{align}\nonumber
\left|c_1^{[1]}(\alpha,n)\right|&=\left|3+3\alpha-2\log 8-2(1+\alpha)\log n\right|\\\nonumber
&=2(1+\alpha)\log n-(3+3\alpha-2\log 8)\\\nonumber
&\hspace{2 cm} \left(\text{as for}\ n\ge \left\lceil e^{\frac{3+3\alpha-2\log 8}{2+2\alpha}}\right\rceil,\ 2(1+\alpha)\log n \ge 3+3\alpha-2\log 8 \right)\\\label{Gnewcoeff1estim2}
&<n^{\frac{1}{16}},
\end{align}
where in the last line we used Lemma \ref{newlemma2} with $(a,b)=(2\log8-3-3\alpha, 2+2\alpha)$, we get for all $n\ge \max \left\{\left\lceil e^{\frac{2\log 8-3-3\alpha}{2+2\alpha}}\right\rceil, \left\lceil (4+4\alpha)^{227}\right\rceil \right\}$, $2(1+\alpha)\log n-(3+3\alpha-2\log 8)<n^{\frac{1}{16}}$.

Combining \eqref{Gnewcoeff1estim1} and \eqref{Gnewcoeff1estim2}, for all
\begin{align*}
n&\ge \max\left\{ \left\lceil e^{\frac{3+3\alpha-2\log 8}{2+2\alpha}}\right\rceil, \left\lceil e^{\frac{2\log 8-3-3\alpha}{2+2\alpha}}\right\rceil, \left\lceil (4+4\alpha)^{227}\right\rceil,(2.4)^{16}\right\}\\
&=\max\left\{ \left\lceil e^{\frac{3+3\alpha-2\log 8}{2+2\alpha}}\right\rceil, \left\lceil e^{\frac{2\log 8-3-3\alpha}{2+2\alpha}}\right\rceil, \left\lceil (4+4\alpha)^{227}\right\rceil\right\}\ \left(\text{as}\ (4+4\alpha)^{227}>4^{227}>(2.4)^{16}\right)\\
&:=N_2^{[1]}(\alpha),
\end{align*}
we get
\begin{equation}\label{Gnewcoeff1estimfinal}
\underset{0\le k\le 4}{\max}\{|c_k^{[1]}(\alpha,n)|\}<n^{\frac{1}{16}}.
\end{equation}

Similarly, following \eqref{Gnewcoeff2}, by numercial computations with Mathematica, it follows that
\begin{equation}\label{Gnewcoeff2estim1}
\underset{k\in\{0,2,4\}}{\max}\{|c_k^{[2]}(\alpha,n)|\}<16.2.
\end{equation}
Since $c_1^{[2]}(\alpha,n)=c_1^{[1]}(\alpha,n)$, from \eqref{Gnewcoeff1estim2}, for all $$n\ge \max\left\{ \left\lceil e^{\frac{3+3\alpha-2\log 8}{2+2\alpha}}\right\rceil, \left\lceil e^{\frac{2\log 8-3-3\alpha}{2+2\alpha}}\right\rceil, \left\lceil (4+4\alpha)^{227}\right\rceil \right\},$$
it follows that
\begin{equation}\label{Gnewcoeff2estim2}
\left|c_1^{[2]}(\alpha,n)\right|<n^{\frac{1}{16}}.
\end{equation}
Finally for $k=3$,
\begin{align}\nonumber
\left|c_3^{[2]}(\alpha,n)\right|&=\left|\frac{3+11\pi^2+11\alpha\pi^2-6\pi^2\log 8-(6+6\alpha)\pi^2\log n}{\pi^2}\right|\\\nonumber
&=6(1+\alpha)\log n-\left(\frac{3}{\pi^2}+11(1+\alpha)-6\log 8\right)\\\nonumber
&\hspace{2 cm} \left(\text{as for}\ n\ge \left\lceil e^{\frac{\frac{3}{\pi^2}+11(1+\alpha)-6\log 8}{6+6\alpha}}\right\rceil,\ 6(1+\alpha)\log n \ge \frac{3}{\pi^2}+11(1+\alpha)-6\log 8 \right)\\\label{Gnewcoeff2estim3}
&<n^{\frac{1}{16}},
\end{align}
where in the last line we used Lemma \ref{newlemma2} with $(a,b)=(6\log 8-11(1+\alpha)-\frac{3}{\pi^2}, 6+6\alpha)$, we get for all $n\ge \max \left\{\left\lceil e^{\frac{6\log 8-11(1+\alpha)-\frac{3}{\pi^2}}{6+6\alpha}}\right\rceil, \left\lceil (12+12\alpha)^{227}\right\rceil \right\}$, $6(1+\alpha)\log n-\left(\frac{3}{\pi^2}+11(1+\alpha)-6\log 8\right)<n^{\frac{1}{16}}$.

Therefore, from \eqref{Gnewcoeff2estim1}-\eqref{Gnewcoeff2estim3}, for all
\begin{align*}
n&\ge \max\left\{ \left\lceil e^{\frac{3+3\alpha-2\log 8}{2+2\alpha}}\right\rceil, \left\lceil e^{\frac{2\log 8-3-3\alpha}{2+2\alpha}}\right\rceil, \left\lceil e^{\frac{\frac{3}{\pi^2}+11(1+\alpha)-6\log 8}{6+6\alpha}}\right\rceil, \left\lceil e^{\frac{6\log 8-11(1+\alpha)-\frac{3}{\pi^2}}{6+6\alpha}}\right\rceil, \left\lceil (12+12\alpha)^{227}\right\rceil \right\}\\
&=:N_2^{[2]}(\alpha),
\end{align*}
it follows that
\begin{equation}\label{Gnewcoeff2estimfinal}
\underset{0\le k\le 4}{\max}\{|c_k^{[2]}(\alpha,n)|\}<n^{\frac{1}{16}}.
\end{equation}

Finally from \eqref{Gnewcoeff1estimfinal} and \eqref{Gnewcoeff2estimfinal}, for all $n\ge \max\{N_2^{[1]}(\alpha), N_2^{[2]}(\alpha)\}=:N_2(\alpha)$, we get 
\begin{equation}\label{Gnewcoeffestimfinal}
\underset{0\le k\le 4}{\max}\{|c_k^{[1]}(\alpha,n)|,|c_k^{[2]}(\alpha,n)|\}<n^{\frac{1}{16}}.
\end{equation}

Now we present the proof of Theorem \ref{mainresult3} using \eqref{Gneweqn1}-\eqref{Gneweqn4}.
\subsection{Proof of Theorem \ref{mainresult3}}\label{prof2}
	Rewrite \eqref{mainresult3eqn} in terms of $u_{\alpha}(n)$ as follows
	\begin{equation}\label{maineqn}
		4(1-u_{\alpha}(n))(1-u_{\alpha}(n+1))<(1-u_{\alpha}(n) u_{\alpha}(n+1))^2.
	\end{equation}
From Theorem \ref{newtheorem}, it is clear that $u_{\alpha}(n)>1$ and so is $u_{\alpha}(n+1)$. Therefore, in order to prove \eqref{maineqn}, it is equivalent to show that
\begin{equation*}
4(u_{\alpha}(n)-1)(u_{\alpha}(n+1)-1)<(u_{\alpha}(n) u_{\alpha}(n+1)-1)^2.
\end{equation*}
From Theorem \ref{newtheorem} and \eqref{Gneweqn1}-\eqref{Gneweqn4},  we know that for all $n\ge \max\{\widetilde{N}(\alpha), N_1(\alpha)\}$,
$$4(u_{\alpha}(n)-1)(u_{\alpha}(n+1)-1)<4(U(n,\alpha)-1)(U_1(n,\alpha)-1)<4(U^*(n,\alpha)-1)(U_1^*(n,\alpha)-1)$$
and 
$$(u_{\alpha}(n) u_{\alpha}(n+1)-1)^2>(L(n,\alpha)L_1(n,\alpha)-1)^2>(L^*(n,\alpha)L_1^*(n,\alpha)-1)^2.$$
Therefore to prove \eqref{maineqn}, it is sufficient to prove that
\begin{equation}\label{Gneweqn5}
4(U^*(n,\alpha)-1)(U_1^*(n,\alpha)-1)<(L^*(n,\alpha)L_1^*(n,\alpha)-1)^2.
\end{equation}
Following \eqref{Gneweqn1} and \eqref{Gneweqn3}, it follows that
\begingroup
\allowdisplaybreaks
\begin{align}\nonumber
&(U^*(n,\alpha)-1)(U_1^*(n,\alpha)-1)\\\nonumber
&=\left(\frac{1}{n^{\frac 52}}\sum_{k=0}^{4}\frac{c_k^{[1]}(\alpha,n)}{\sqrt{n}^k}+\frac{1}{n^{\frac{19}{4}}}\right)\left(\frac{1}{n^{\frac 52}}\sum_{k=0}^{4}\frac{c_k^{[2]}(\alpha,n)}{\sqrt{n}^k}+\frac{1}{n^{\frac{19}{4}}}\right)\\\nonumber
&=\frac{1}{n^5}\cdot\sum_{k=0}^{4}\frac{c_k^{[1]}(\alpha,n)}{\sqrt{n}^k}\cdot \sum_{k=0}^{4}\frac{c_k^{[2]}(\alpha,n)}{\sqrt{n}^k}+\frac{1}{n^{\frac{29}{4}}}\left(\sum_{k=0}^{4}\frac{c_k^{[1]}(\alpha,n)+c_k^{[2]}(\alpha,n)}{\sqrt{n}^k}\right)+\frac{1}{n^{\frac{19}{2}}}\\\nonumber
&=\frac{1}{n^5}\left(\sum_{k=0}^{4}\sum_{m=0}^{k}\frac{c_m^{[1]}(\alpha,n)\cdot c_{k-m}^{[2]}(\alpha,n)}{\sqrt{n}^k}+\frac{1}{n^{\frac 52}}\sum_{k=0}^{3}\sum_{m=k}^{3}\frac{c_{m+1}^{[1]}(\alpha,n)\cdot c_{4+k-m}^{[2]}(\alpha,n)}{\sqrt{n}^k}\right)\\\nonumber
&\hspace{6.5 cm}+\frac{1}{n^{\frac{29}{4}}}\left(\sum_{k=0}^{4}\frac{c_k^{[1]}(\alpha,n)+c_k^{[2]}(\alpha,n)}{\sqrt{n}^k}\right)+\frac{1}{n^{\frac{19}{2}}}\\\nonumber
&=\frac{1}{n^5}\sum_{k=0}^{4}\sum_{m=0}^{k}\frac{c_m^{[1]}(\alpha,n)\cdot c_{k-m}^{[2]}(\alpha,n)}{\sqrt{n}^k}+\frac{1}{n^{\frac{15}{2}}}\sum_{k=0}^{3}\sum_{m=k}^{3}\frac{c_{m+1}^{[1]}(\alpha,n)\cdot c_{4+k-m}^{[2]}(\alpha,n)}{\sqrt{n}^k}\\\nonumber
&\hspace{6.5 cm}+\frac{1}{n^{\frac{29}{4}}}\left(\sum_{k=0}^{4}\frac{c_k^{[1]}(\alpha,n)+c_k^{[2]}(\alpha,n)}{\sqrt{n}^k}\right)+\frac{1}{n^{\frac{19}{2}}}\\\label{Gneweqn6}
&=:\mathcal{M}_1(n,\alpha)+\mathcal{E}_1^{[1]}(n,\alpha)+\mathcal{E}_1^{[2]}(n,\alpha)+\frac{1}{n^{\frac{19}{2}}}.
\end{align}
\endgroup
First we give an upper bound for $\left|\mathcal{E}_1^{[1]}(n,\alpha)\right|$ as
\begin{align}\nonumber
\left|\mathcal{E}_1^{[1]}(n,\alpha)\right|&\le \frac{1}{n^{\frac{15}{2}}}\sum_{k=0}^{3}\sum_{m=k}^{3}\frac{|c_{m+1}^{[1]}(\alpha,n)|\cdot |c_{4+k-m}^{[2]}(\alpha,n)|}{\sqrt{n}^k}\\\nonumber
&\le \frac{1}{n^{\frac{15}{2}}}\sum_{k=0}^{3}\sum_{m=k}^{3}\frac{n^{\frac 18}}{\sqrt{n}^k}\ \ \left(\text{by}\ \eqref{Gnewcoeffestimfinal}\right)\ \le \frac{1}{n^{\frac{15}{2}}}\sum_{k=0}^{3}\sum_{m=k}^{3}\frac{n^{\frac 14}}{\sqrt{n}^k} \\\label{Gneweqn7}
&\le \frac{1}{n^{\frac{15}{2}}}\sum_{k=0}^{3}\sum_{m=k}^{3}n^{\frac 14}\ \ \left(\text{as}\ n\ge 1\right)=\frac{10}{n^{\frac{29}{4}}}.
\end{align}
Next we estimate an upper bound for $\left|\mathcal{E}_2^{[1]}(n,\alpha)\right|$ as
\begin{align}\nonumber
\left|\mathcal{E}_2^{[1]}(n,\alpha)\right|&\le \frac{1}{n^{\frac{29}{4}}}\left(\sum_{k=0}^{4}\frac{|c_k^{[1]}(\alpha,n)|+|c_k^{[2]}(\alpha,n)|}{\sqrt{n}^k}\right)\\\nonumber
&\le \frac{1}{n^{\frac{29}{4}}}\sum_{k=0}^{4}\frac{2\cdot n^{\frac{1}{16}}}{\sqrt{n}^k}\ \ \left(\text{by}\ \eqref{Gnewcoeffestimfinal}\right)\ \le \frac{2}{n^{\frac{29}{4}}}\sum_{k=0}^{4}n^{\frac 18}\ \ \left(\text{as}\ n\ge 1\right)\\\label{Gneweqn8}
&=\frac{10}{n^{\frac{57}{8}}}.
\end{align}
Applying \eqref{Gneweqn7} and \eqref{Gneweqn8} to \eqref{Gneweqn6}, we get for all $n\ge \max\{N_2(\alpha), 257\}$,
\begin{align}\label{GTfinalupperbound}
4(U^*(n,\alpha)-1)(U_1^*(n,\alpha)-1)< 4\mathcal{M}_1(n,\alpha)+4\left(\frac{10}{n^{\frac{29}{4}}}+\frac{10}{n^{\frac{57}{8}}}+\frac{1}{n^{\frac{19}{2}}}\right)<4\mathcal{M}_1(n,\alpha)+\frac{60}{n^{\frac{57}{8}}},
\end{align}
where in the last inequality we have used the fact that for all $n\ge 257$, $4\left(\frac{10}{n^{\frac{29}{4}}}+\frac{10}{n^{\frac{57}{8}}}+\frac{1}{n^{\frac{19}{2}}}\right)<\frac{60}{n^{\frac{57}{8}}}$.
Following \eqref{Gneweqn2} and \eqref{Gneweqn4}, it follows that
\begingroup
\allowdisplaybreaks
\begin{align}\nonumber
&L^*(n,\alpha)\cdot L_1^*(n,\alpha)-1\\\nonumber
&=\left(1+\frac{1}{n^{\frac 52}}\sum_{k=0}^{4}\frac{c_k^{[1]}(\alpha,n)}{\sqrt{n}^k}-\frac{1}{n^{\frac{19}{4}}}\right)\left(1+\frac{1}{n^{\frac 52}}\sum_{k=0}^{4}\frac{c_k^{[2]}(\alpha,n)}{\sqrt{n}^k}-\frac{1}{n^{\frac{19}{4}}}\right)-1\\\nonumber
&=\frac{1}{n^{\frac 52}}\sum_{k=0}^{4}\frac{c_k^{[1]}(\alpha,n)+c_k^{[2]}(\alpha,n)}{\sqrt{n}^k}+\frac{1}{n^5}\cdot \sum_{k=0}^{4}\frac{c_k^{[1]}(\alpha,n)}{\sqrt{n}^k}\cdot \sum_{k=0}^{4}\frac{c_k^{[1]}(\alpha,n)}{\sqrt{n}^k}\\\nonumber
&\hspace{5.5cm}-\frac{1}{n^{\frac{29}{4}}}\sum_{k=0}^{4}\frac{c_k^{[1]}(\alpha,n)+c_k^{[2]}(\alpha,n)}{\sqrt{n}^k}-\frac{2}{n^{\frac{19}{4}}}+\frac{1}{n^{\frac{19}{2}}}\\\nonumber
&=\frac{1}{n^{\frac 52}}\sum_{k=0}^{4}\frac{c_k^{[1]}(\alpha,n)+c_k^{[2]}(\alpha,n)}{\sqrt{n}^k}+\frac{1}{n^5}\sum_{k=0}^{4}\sum_{m=0}^{k}\frac{c_m^{[1]}(\alpha,n)\cdot c_{k-m}^{[2]}(\alpha,n)}{\sqrt{n}^k}\\\nonumber
&\hspace{1 cm}+\frac{1}{n^{\frac{15}{2}}}\sum_{k=0}^{3}\sum_{m=k}^{3}\frac{c_{m+1}^{[1]}(\alpha,n)\cdot c_{4+k-m}^{[2]}(\alpha,n)}{\sqrt{n}^k}-\frac{1}{n^{\frac{29}{4}}}\sum_{k=0}^{4}\frac{c_k^{[1]}(\alpha,n)+c_k^{[2]}(\alpha,n)}{\sqrt{n}^k}-\frac{2}{n^{\frac{19}{4}}}+\frac{1}{n^{\frac{19}{2}}}\\\label{Gneweqn9}
&=:\mathcal{M}_2^{[1]}(n,\alpha)+\mathcal{E}_1^{[1]}(n,\alpha)-\mathcal{E}_1^{[2]}(n,\alpha)-\frac{2}{n^{\frac{19}{4}}}+\frac{1}{n^{\frac{19}{2}}},
\end{align}
\endgroup
where
\begin{align}\nonumber
\mathcal{M}_2^{[1]}(n,\alpha)&=\frac{1}{n^{\frac 52}}\sum_{k=0}^{4}\frac{c_k^{[1]}(\alpha,n)+c_k^{[2]}(\alpha,n)}{\sqrt{n}^k}+\frac{1}{n^5}\sum_{k=0}^{4}\sum_{m=0}^{k}\frac{c_m^{[1]}(\alpha,n)\cdot c_{k-m}^{[2]}(\alpha,n)}{\sqrt{n}^k}\\\nonumber
&=\frac{1}{n^{\frac 52}}\left(\sum_{k=0}^{4}\frac{c_k^{[1]}(\alpha,n)+c_k^{[2]}(\alpha,n)}{\sqrt{n}^k}+\sum_{k=0}^{4}\sum_{m=0}^{k}\frac{c_m^{[1]}(\alpha,n)\cdot c_{k-m}^{[2]}(\alpha,n)}{\sqrt{n}^{k+5}}\right)\\\nonumber
&=\frac{1}{n^{\frac 52}}\left(\sum_{k=0}^{4}\frac{c_k^{[1]}(\alpha,n)+c_k^{[2]}(\alpha,n)}{\sqrt{n}^k}+\sum_{k=5}^{9}\sum_{m=5}^{k}\frac{c_{m-5}^{[1]}(\alpha,n)\cdot c_{k-m}^{[2]}(\alpha,n)}{\sqrt{n}^{k}}\right)\\\label{Gneweqn10}
&=:\frac{1}{n^{\frac 52}}\sum_{k=0}^{9}\frac{d_k^{[1]}(\alpha,n)}{\sqrt{n}^k},
\end{align}
with
\begin{equation}\label{Gneweqn11}
d_k^{[1]}(\alpha,n)=
\begin{cases}
c_k^{[1]}(\alpha,n)+c_k^{[2]}(\alpha,n), &\quad \text{if}\ 0\le k\le 4,\\
\displaystyle\sum_{m=5}^{k}c_{m-5}^{[1]}(\alpha,n)\cdot c_{k-m}^{[2]}(\alpha,n), &\quad \text{if}\ 5\le k\le 9.
\end{cases}
\end{equation}
Applying \eqref{Gneweqn7} and \eqref{Gneweqn8} to \eqref{Gneweqn10}, we have for all $n\ge \max\{N_2(\alpha), 4\}$,
\begin{align}\label{GTlowerbound}
L^*(n,\alpha)\cdot L_1^*(n,\alpha)-1> \mathcal{M}_2^{[1]}(n,\alpha)-\left(\frac{10}{n^{\frac{29}{4}}}+\frac{10}{n^{\frac{57}{8}}}+\frac{2}{n^{\frac{19}{4}}}-\frac{1}{n^{\frac{19}{2}}}\right)>\mathcal{M}_2^{[1]}(n,\alpha)-\frac{3}{n^{\frac{19}{4}}},
\end{align}
where in the last inequality we have used the fact that for all $n\ge 4$, $\frac{10}{n^{\frac{29}{4}}}+\frac{10}{n^{\frac{57}{8}}}+\frac{2}{n^{\frac{19}{4}}}-\frac{1}{n^{\frac{19}{2}}}<\frac{3}{n^{\frac{19}{4}}}$.

Applying \eqref{Gnewcoeffestimfinal} to \eqref{Gneweqn11}, we get
\begin{equation*}
\underset{0\le k\le 4}{\max}\left\{\left|d_k^{[1]}(\alpha,n)\right|\right\}\le 2\cdot n^{\frac{1}{16}}\ \ \text{and}\ \ \underset{5\le k\le 9}{\max}\left\{\left|d_k^{[1]}(\alpha,n)\right|\right\}\le 5\cdot n^{\frac 18},
\end{equation*}
and so
\begin{equation}\label{Gneweqn12}
\underset{0\le k\le 9}{\max}\left\{\left|d_k^{[1]}(\alpha,n)\right|\right\}\le 5\cdot n^{\frac 18}.
\end{equation}
Now we estimate $\left(\mathcal{M}_2^{[1]}(n,\alpha)-\frac{3}{n^{\frac{19}{4}}}\right)^2$ by splitting it in the following way
\begingroup
\allowdisplaybreaks
\begin{align}\nonumber
&\left(\mathcal{M}_2^{[1]}(n,\alpha)-\frac{3}{n^{\frac{19}{4}}}\right)^2\\\nonumber
&=\left(\mathcal{M}_2^{[1]}(n,\alpha)\right)^2-\frac{6}{n^{\frac{19}{4}}}\cdot \mathcal{M}_2^{[1]}(n,\alpha)+\frac{9}{n^{\frac{19}{2}}}\\\nonumber
&=\frac{1}{n^5}\cdot \sum_{k=0}^{9}\frac{d_k^{[1]}(\alpha,n)}{\sqrt{n}^k}\cdot \sum_{k=0}^{9}\frac{d_k^{[1]}(\alpha,n)}{\sqrt{n}^k}-\frac{6}{n^{\frac{29}{4}}}\cdot\sum_{k=0}^{9}\frac{d_k^{[1]}(\alpha,n)}{\sqrt{n}^k}+\frac{9}{n^{\frac{19}{2}}}\\\nonumber
&=\frac{1}{n^5}\left(\sum_{k=0}^{9}\sum_{m=0}^{k}\frac{d_m^{[1]}(\alpha,n)\cdot d_{k-m}^{[1]}(\alpha,n)}{\sqrt{n}^k}+\frac{1}{n^5}\cdot \sum_{k=0}^{8}\sum_{m=k}^{8}\frac{d_{m+1}^{[1]}(\alpha,n)\cdot d_{9+k-m}^{[1]}(\alpha,n)}{\sqrt{n}^k}\right)\\\nonumber
&\hspace{9 cm}-\frac{6}{n^{\frac{29}{4}}}\cdot\sum_{k=0}^{9}\frac{d_k^{[1]}(\alpha,n)}{\sqrt{n}^k}+\frac{9}{n^{\frac{19}{2}}}\\\nonumber
&=\frac{1}{n^5}\sum_{k=0}^{4}\sum_{m=0}^{k}\frac{d_m^{[1]}(\alpha,n)\cdot d_{k-m}^{[1]}(\alpha,n)}{\sqrt{n}^k}+\frac{1}{n^5}\sum_{k=5}^{9}\sum_{m=0}^{k}\frac{d_m^{[1]}(\alpha,n)\cdot d_{k-m}^{[1]}(\alpha,n)}{\sqrt{n}^k}\\\nonumber
&\hspace{2 cm} +\frac{1}{n^{10}}\cdot \sum_{k=0}^{8}\sum_{m=k}^{8}\frac{d_{m+1}^{[1]}(\alpha,n)\cdot d_{9+k-m}^{[1]}(\alpha,n)}{\sqrt{n}^k}-\frac{6}{n^{\frac{29}{4}}}\cdot\sum_{k=0}^{9}\frac{d_k^{[1]}(\alpha,n)}{\sqrt{n}^k}+\frac{9}{n^{\frac{19}{2}}}\\\label{Gneweqn13}
&=:\mathcal{M}_2(n,\alpha)+\mathcal{E}_2^{[1]}(n,\alpha)+\mathcal{E}_2^{[2]}(n,\alpha)-\mathcal{E}_2^{[3]}(n,\alpha)+\frac{9}{n^{\frac{19}{2}}}.
\end{align}
\endgroup
Estimating an upper bound for $\left|\mathcal{E}_2^{[1]}(n,\alpha)\right|$ as
\begin{align}\nonumber
\left|\mathcal{E}_2^{[1]}(n,\alpha)\right|\le \frac{1}{n^5}\sum_{k=5}^{9}\sum_{m=0}^{k}\frac{\left|d_m^{[1]}(\alpha,n)\right|\cdot \left|d_{k-m}^{[1]}(\alpha,n)\right|}{\sqrt{n}^k}&\le \frac{1}{n^{\frac{15}{2}}}\sum_{k=5}^{9}\sum_{m=0}^{k}\left|d_m^{[1]}(\alpha,n)\right|\cdot \left|d_{k-m}^{[1]}(\alpha,n)\right|\\\nonumber
&\le \frac{25}{n^{\frac{15}{2}}}\sum_{k=5}^{9}\sum_{m=0}^{k}n^{\frac 14}\ \ \left(\text{by}\ \eqref{Gneweqn12}\right)\\\label{Gneweqn14}
&=\frac{1375}{n^{\frac{29}{4}}}.
\end{align} 
Next, we compute an upper bound for $\left|\mathcal{E}_2^{[2]}(n,\alpha)\right|$ as
\begin{align}\nonumber
\left|\mathcal{E}_2^{[2]}(n,\alpha)\right|&\le \frac{1}{n^{10}}\sum_{k=0}^{8}\sum_{m=k}^{8}\frac{\left|d_{m+1}^{[1]}(\alpha,n)\right|\cdot \left|d_{9+k-m}^{[1]}(\alpha,n)\right|}{\sqrt{n}^k}\\\nonumber
&\le \frac{1}{n^{10}}\sum_{k=0}^{8}\sum_{m=k}^{8}\left|d_{m+1}^{[1]}(\alpha,n)\right|\cdot \left|d_{9+k-m}^{[1]}(\alpha,n)\right|\le \frac{25}{n^{10}}\sum_{k=0}^{8}\sum_{m=k}^{8}n^{\frac 14}\ \ \left(\text{by}\ \eqref{Gneweqn12}\right)\\\label{Gneweqn15}
&=\frac{1125}{n^{\frac{39}{4}}}.
\end{align} 
Finally, for $\left|\mathcal{E}_2^{[3]}(n,\alpha)\right|$, we get
\begin{equation}\label{Gneweqn16}
\left|\mathcal{E}_2^{[3]}(n,\alpha)\right|\le \frac{6}{n^{\frac{29}{4}}}\cdot\sum_{k=0}^{9}\frac{\left|d_k^{[1]}(\alpha,n)\right|}{\sqrt{n}^k}\le \frac{30}{n^{\frac{29}{4}}}\cdot\sum_{k=0}^{9}n^{\frac 18}=\frac{300}{n^{\frac{57}{8}}}.
\end{equation}
Applying \eqref{Gneweqn14}-\eqref{Gneweqn16} to \eqref{Gneweqn13}, it follows that for all $n\ge \max\{N_2(\alpha), 761\}$
\begin{equation*}
\left(\mathcal{M}_2^{[1]}(n,\alpha)-\frac{3}{n^{\frac{19}{4}}}\right)^2>\mathcal{M}_2(n,\alpha)-\frac{1375}{n^{\frac{29}{4}}}-\frac{1125}{n^{\frac{39}{4}}}-\frac{300}{n^{\frac{57}{8}}}+\frac{9}{n^{\frac{19}{2}}}>\mathcal{M}_2(n,\alpha)-\frac{900}{n^{\frac{57}{8}}},
\end{equation*}
and consequently from \eqref{GTlowerbound}, we have for all $n \ge \max\{N_2(\alpha), 761\}$,
\begin{align}\label{GTfinallowerbound}
(L^*(n,\alpha)\cdot L_1^*(n,\alpha)-1)^2> \mathcal{M}_2(n,\alpha)-\frac{900}{n^{\frac{57}{8}}}.
\end{align}  
Finally, to prove \eqref{Gneweqn5}, from \eqref{GTfinalupperbound} and \eqref{GTfinallowerbound}, it remains to show that 
\begin{equation}\label{Gneweqn17}
\mathcal{M}_2(n,\alpha)-4\mathcal{M}_1(n,\alpha)>\frac{960}{n^{\frac{57}{8}}}.
\end{equation}
Computing with Mathematica, we have checked that $\mathcal{M}_2(n,\alpha)-4\mathcal{M}_1(n,\alpha)=\frac{225\pi^2}{64n^7}$,
and hence for all $n\ge 343361460986$, \eqref{Gneweqn17} holds. Therefore, for all
\begin{equation}\label{Gfinalcutoff}
n\ge \max\left\{\widetilde{N}(\alpha), N_1(\alpha), N_2(\alpha),343361460986\right\}=:N_T(\alpha),
\end{equation} 
we get \eqref{Gneweqn5} and hence conclude the proof of \eqref{maineqn}.

\section{Conclusion}\label{conclusion}
We conclude this paper by considering the following problem:
\begin{problem}\label{Problem}
Let $\alpha$ be a non-negative real number. Does the sequence $\Bigl\{\sqrt[n]{\overline{p}(n)/n^{\alpha}}\Bigr\}_{n\geq N(\alpha)}$ is infinitely $\log$-convex? Moreover, if it is infinitely $\log$-convex, then what about the growth of $N(\alpha)$?	
\end{problem}
\begin{center}
	\textbf{Acknowledgements}
\end{center}
 Part of this work has been carried out in National Institute of Science Education and Research. The author would like to thank the institute for its hospitality and support. The author expresses her sincere gratitude to Prof. Brundaban Sahu for his valuable suggestions on improving the manuscript of the paper.\\
 \textbf{Data availability statement}: We hereby confirm that Data sharing not applicable to this article as no datasets were generated or analyzed during the current study.

	\end{document}